\documentclass[12pt]{article}

\usepackage{stmaryrd}
\usepackage{mathrsfs}
\usepackage{amsmath,amsthm,amssymb}
\usepackage[colorlinks,
            linkcolor=red,
            anchorcolor=blue,
            citecolor=magenta
            ]{hyperref}
\usepackage{pdfsync}
\usepackage[numbers,sort&compress]{natbib}
\allowdisplaybreaks

\theoremstyle{definition}

\theoremstyle{plain}
\newtheorem{theorem}{Theorem}[section]
\newtheorem{definition}{Definition}[section]
\newtheorem{remark}{Remark}[section]
\newtheorem{lemma}{Lemma}[section]

\newtheorem{proposition}{Proposition}[section]

\numberwithin{equation}{section}

\newcommand{\vs}{\vspace}

\begin{document}

\title{Almost sure well-posedness and orbital stability for Schr\"odinger equation with potential\footnote{  This work was partially supported by NNSFC (No. 12171493).}}

\author{ Jun Wang$^{a}$, Zhaoyang Yin$^{a, b}$\footnote {Corresponding author. wangj937@mail2.sysu.edu.cn (J. Wang), mcsyzy@mail.sysu.edu.cn (Z. Yin)
} \\
%EndAName
{\small $^{a}$Department of Mathematics, Sun Yat-sen University, Guangzhou, 510275, China } \\
{\small $^{b}$School of Science, Shenzhen Campus of Sun Yat-sen University, Shenzhen, 518107, China } \\
}

	\date{}

	\maketitle

\date{}

 \maketitle \vs{-.7cm}

  \begin{abstract}
In this paper, we
study the almost sure well-posedness theory and orbital stability for the nonlinear Schr\"odinger equation with potential
\begin{equation*}
  \left\{\begin{array}{l}
i \partial_t u+\Delta u-V(x)u+|u|^{2}u=0,\ (x, t) \in \mathbb{R}^4 \times \mathbb{R}, \\
\left.u\right|_{t=0}=f \in  H ^s(\mathbb{R}^4),
\end{array}\right.
\end{equation*}
where $\frac{1}{3}<s<1$ and $V(x):\mathbb{R}^4\rightarrow \mathbb{R}$ satisfies appropriate conditions. The main idea in the proofs is  based on Strichartz spaces
as well as variants of local smoothing, inhomogeneous local
smoothing and maximal function spaces. To our best knowledge, this is the first orbital stability result for this model.
 \end{abstract}

{\footnotesize {\bf   Keywords:} Schr\"odinger equation, Almost sure well-posedness, Orbital stability, Random initial data

}

\section{ Introduction and main results}

This paper studies the almost sure well-posedness theory and orbital stability for the nonlinear Schr\"odinger equation with potential
\begin{equation}\label{eq1.1}
\left\{\begin{array}{l}
i \partial_t u+\Delta u-V(x)u +|u|^{2}u=0,\ (x, t) \in \mathbb{R}^4 \times \mathbb{R}, \\
\left.u\right|_{t=0}=f  \in  H ^s(\mathbb{R}^4),
\end{array}\right.
\end{equation}
where $V(x):\mathbb{R}^4\rightarrow \mathbb{R}$ satisfies appropriate conditions.

Consider the following equation
\begin{equation*}
\left\{\begin{array}{l}
i \partial_t u+\Delta u  =\pm|u|^{2}u,\ (x, t) \in \mathbb{R}^4 \times \mathbb{R}, \\
\left.u\right|_{t=0}=f  \in  H ^s(\mathbb{R}^4).
\end{array}\right.
\end{equation*}
For this model under the probabilistic setting, the local well-posedness, small data scattering, and conditional global well-posedness results have been established before. This approach was initiated by Bourgain \cite{{JB1994},{JB1996}} for the periodic nonlinear Schr\"odinger equation in one and two space dimensions, building
upon the constructions of invariant measures by Burq-Tzvetkov \cite{{NBI2008},{NBII2008}} in the context of the cubic nonlinear wave
equation on a three-dimensional compact Riemannian manifold. There has since been a
vast and fascinating body of research, using probabilistic tools to study many nonlinear
dispersive or hyperbolic equations in scaling super-critical regimes, see for example \cite{{JBAB2014},{NBNT2014},{JCTO2012},{YD2012}} and references therein. Previously for the random data local result, Pocovnicu and
Wang's argument \cite{OPYW2018} is based on the dispersive inequality; Dodson et. al \cite{BDJL2019} used the high dimensional version of smoothing effect and maximal function estimates. Recently, \cite{JSAS2023} give some simple new approaches combining the atom space method by Koch-Tataru \cite{HKDT2005} and the variants of bilinear Strichartz estimate.

In this paper, we are interested in the almost sure well-posedness of the Schr\"odinger equation with potential and the orbital stability of the ground state. The potential $V: \mathbb{R}^d \rightarrow \mathbb{R}$ is assumed to satisfy the following assumptions
\begin{equation}\label{eq1.2}
V \in \mathcal{K} \cap L^{\frac{d}{2}}
\end{equation}
and
\begin{equation}\label{eq1.3}
\left\|V_{-}\right\|_{\mathcal{K}}<d(d-2)\alpha(d),
\end{equation}
where $\alpha(d)$ denotes the volume of the unit ball in $\mathbb{R}^d(d\geq3)$, $\mathcal{K}$ is a class of Kato potentials with
$$
\|V\|_{\mathcal{K}}:=\sup\limits_{x \in \mathbb{R}^d} \int_{\mathbb{R}^d} \frac{|V(y)|}{|x-y|} d y
$$
and $V_{-}(x):=\min \{V(x), 0\}$ is the negative part of $V$. Conditions \eqref{eq1.2} and \eqref{eq1.3} are the key to obtain Strichartz estimates, see \cite{JWZY2024}.

The main results of this paper are as follows.
\begin{theorem}\label{t1.1}
Let $\frac{1}{3}<s<1$. Let $f \in H_x^s(\mathbb{R}^4)$ and denote by $f^\omega$ the randomization of $f$ as defined in \eqref{eq2.99}. Then for almost every $\omega \in \Omega$ there exists an open interval $I(0\in I)$ and a unique solution
$$
u(t) \in e^{i t H} f^\omega+C(I ; H_x^1(\mathbb{R}^4))
$$
to the cubic nonlinear Schr\"odinger equation
\begin{equation}\label{eq1.4}
  \left\{\begin{aligned}
&(i \partial_t+H) u+|u|^2 u=0 \text { on } I \times \mathbb{R}^4, \\
&u(0)  =f^\omega .
\end{aligned}\right.
\end{equation}
\end{theorem}
\begin{remark}\label{R1.1}
$(1)$ In the statement of Theorem \ref{t1.1}, uniqueness holds in the sense that upon writing $u(t)=e^{itH}f^{\omega}+v(t)$, then there exists a unique local solution $v\in C(I ; H_x^1(\mathbb{R}^4))\cap X(\mathbb{R})$ to the forced cubic nonlinear Schr\"odinger equation
\begin{equation*}
  \left\{\begin{aligned}
&(i \partial_t+H) v+|e^{itH}f^{\omega}+v(t)|^2|e^{itH}f^{\omega}+v(t)|=0 \text { on } I \times \mathbb{R}^4, \\
&v(0)  =0.
\end{aligned}\right.
\end{equation*}

$(2)$ The length of the time interval $I$ has to satisfy $\|e^{itH}f^{\omega}\|_{Y (I)}\leq\delta$ for some small absolute constant $0<\delta\ll1$, where the function space $Y(I)$ is defined in Section 3.

$(3)$ Theorem \ref{t1.1} still holds for the case of defocusing.

$(4)$ From the proof of Theorem \ref{t1.1}, it is clear that our methods easily generalize
to other space dimensions $d\geq4$ and to other power-type nonlinearities with potential.
\end{remark}

To achieve this goal, we prove it by using the classical fixed point theorem. However, the construction of function space is complex. The idea of constructing space is based on the Littlewood-Paley projection, which consists of several Strichartz components and a maximum function type component. For the latter, we defined a lateral space and obtained an estimate of the Schr\"odinger propagator in a four-dimensional lateral space, inspired by the local maximum function estimate and local smoothness estimate established by Ionescu-Kenig \cite{AICK2007}. It should be noted that under appropriate potential function conditions, these estimates still hold true. In order to prove Theorem \ref{t1.1}, it is also necessary to estimate the nonlinear term. In fact, we have developed a new trilinear term estimation to handle all possible interaction terms. For this, we need to estimate in high and low frequencies. Firstly, in the case where the deterministic solution $v$ appears at high frequencies, we use Bernstein estimation and Strichartz component estimation in $X$ and $Y$. However, what is more complex is the occurrence of high frequency in $F$, as it has low regularity. Our goal is to put into the local smoothing type $L_{\mathbf{e}_{\ell}}^{\infty, 2}$ component of the $Y$ space. For the remaining low-frequency terms, we only need to use a combination of the Strichartz component and the maximum function component in $X$ and $Y$ for estimation.

\begin{theorem}\label{t1.2}
Let $\frac{1}{2}<s<1$. Let $f \in H_x^s\left(\mathbb{R}^4\right)$ be radially symmetric. Assume $f^\omega$ be the randomized initial data defined in \eqref{eq2.99}. Then for almost every $\omega \in \Omega$ there exists an open interval $I(0\in I)$ and a unique solution
$$
u(t) \in e^{i t H} f^\omega+C(\mathbb{R} ; H_x^1(\mathbb{R}^4))
$$
to the defocusing cubic nonlinear Schr\"odinger equation
\begin{equation*}
  \left\{\begin{aligned}
&(i \partial_t+H) u=|u|^2 u  \text { on } \mathbb{R} \times \mathbb{R}^4, \\
&u(0)  =f^\omega ,
\end{aligned}\right.
\end{equation*}
which scatters as $t\rightarrow\pm\infty$ in the sense that there exist states $v^{\pm}\in H_x^1(\mathbb{R}^4))$ such that
\begin{equation*}
  \lim\limits_{t\rightarrow\pm\infty}\|u(t)-e^{itH}(f^\omega+v^{\pm})\|_{H_x^1(\mathbb{R}^4)}=0.
\end{equation*}
\end{theorem}
\begin{remark}\label{R1.2}
The key to proving Theorem \ref{t1.2} lies in the construction of the function space, as shown in Theorem \ref{t1.1}. The rest follows the techniques in \cite{BDJL2019}, which will not be elaborated here. Note that additional conditions may need to be applied to the potential to make Morawetz estimate hold true. But this is just an exercise, we tend to focus more on orbital stability.
\end{remark}
Now, we focus on the existence of ground state solutions, orbital stability and almost sure orbital stability.
\begin{definition}\label{D1.1}
We say that $u_a\in S(a)$ is a ground state solution to \eqref{eq8.2} if it is a solution having minimal
energy among all the solutions which belong to $S(a)$. Namely, if
\begin{equation*}
  I(u_a)=\inf\{I(u),\ u\in S(a),\ (I\left.\right|_{S(a)}'(u))=0\}.
\end{equation*}
\end{definition}
\begin{definition}\label{D1.2}
$Z\subset \mathcal{H}$ is stable if: $Z\neq\emptyset$ and for any $v\in Z$ and any $\varepsilon>0$, there exists a $\delta>0$ such that if $\varphi\in\mathcal{H}$ satisfies $\|\varphi-v\|_{\mathcal{H}}<\delta$ then $u_{\varphi}$ is globally defined and $\inf\limits_{z\in Z}\|u_{\varphi}(t)-z\|_{\mathcal{H}}<\varepsilon$ for all $t\in\mathbb{R}$, where $u_{\varphi}(t)$ is the solution to \eqref{eq8.1} corresponding to the initial data in $\mathcal{H}$.
\end{definition}

\begin{theorem}\label{t1.3}
Let $d\geq3,\ 2<q<2+\frac{4}{d},\ \|V_{-}\|_{\frac{d}{2}}\leq \mathcal{S}$. There exists a $a_0>0$ such that, for any $a\in(0,a_0)$, $I(u)$ restricted to $S(a)$ has a ground state. This ground state is a (local) minimizer of $I(u)$ in the set
$\mathbf{V}(a)$. In addition, if $(u_n)\subset\mathbf{V} (a)$ is such that $I(u_n)\rightarrow m(a)$, then, up to translation, $u_n \rightarrow u \in \mathcal{M}_a$ in $H^1(\mathbb{R}^d, \mathbb{R})$.
\end{theorem}
\begin{remark}\label{R1.3}
$(1)$ There exists a ground state which is a real valued, radially symmetric decreasing function.
In fact, if $u\in S(a)$ is a ground state then its Schwartz symmetrization is clearly also a ground state.

$(2)$ Under the assumption of Theorem \ref{t1.3}, it can be proved that, for any $a\in(0, a_0)$,
\begin{equation*}
  \mathcal{M}_a=\{e^{i\theta}u \text{ for some } \theta\in\mathbb{R}, u\in S(a)\cap H^1(\mathbb{R}^d,\mathbb{R}), I(u)=m(a)\}
\end{equation*}
by applying the argument of \cite[Remark 1.4]{{LJJJ2022}}.

$(3)$  Under the assumption of Theorem \ref{t1.3}, additionally, $V\in C^1(\mathbb{R}^d)\cap L^{\frac{d}{2}}(\mathbb{R}^d)$ is bounded and $\nabla V(x)\cdot x$ is bounded. If $u\in S(a)$ is a ground state then the associated Lagrange multiplier $\lambda\in\mathbb{R}$ in \eqref{eq8.2} satisfies $\lambda<0$, see \cite{TBAQ2023}.
\end{remark}
\begin{theorem}\label{t1.4}
Let $d\geq3,\ 2<q<2+\frac{4}{d},\ \|V_{-}\|_{\frac{d}{2}}\leq \mathcal{S}$ and $a_0>0$ be given in Theorem \ref{t1.3}. Then for any $a\in(0,a_0)$, the set $\mathcal{M}_a$ is compact, up to translation, and it is orbitally stable.
\end{theorem}
\begin{remark}\label{R1.4}
According to the definition of orbitally stable, this is the direct result of Lemma \ref{L8.9}.
\end{remark}
Inspired by Theorem \ref{t1.2} and Definition \ref{D1.2}, we have the following  definition of almost sure orbitally stable.
\begin{definition}\label{D1.3}
$Z\subset \mathcal{H}$ is almost sure stable if: $Z\neq\emptyset$ and for any $v\in Z$ and any $\varepsilon>0$, there exists a $\delta>0$ such that if $\varphi\in H_x^s(\mathbb{R}^d)$ satisfies $\|\varphi-v\|_{H_x^s(\mathbb{R}^d)}<\delta$ then $u_{\varphi}$ is almost sure globally defined and $\inf\limits_{z\in Z}\|u_{\varphi}(t)-z\|_{\mathcal{H}}<\varepsilon$ for all $t\in\mathbb{R}$, where $u_{\varphi}(t)$ is the almost sure solution to \eqref{eq8.1} corresponding to the random initial data in $H_x^s(\mathbb{R}^d)$.
\end{definition}
\begin{remark}\label{R1.5}
According to the definition of orbitally stable and almost sure orbitally stable, their key difference lies in whether they are globally well-posedness or almost sure globally well-posedness. Obviously, orbital stability can immediately obtain almost sure orbitally stable. However, the introduction of this concept can lead to the development of solutions without orbital stability, which are almost sure orbital stability. In other words, under the influence of random initial data, solutions that were once unstable may become stable.
\end{remark}
In section 2, we provide some notations and some important lemma in the proof of main theorems. In the next section, we aim to construct the functional space required to prove Theorem \ref{t1.1}. In sections 4, we aim to estimate nonlinearity. After that, we establish various almost sure bounds for the free evolution of the random data and prove Theorem \ref{t1.1}. Finally, we establish perturbation theory and obtain orbital stability of ground state solution.

\section{Preliminary}
In this section, we provide some notations and some important lemma in the proof of main theorems. Let $\psi \in C_c^{\infty}(\mathbb{R}^4)$ be an even, non-negative bump function with $\operatorname{supp}(\psi) \subseteq B(0,1)$ and such that
$$
\sum_{k \in \mathbb{Z}^4} \psi(\xi-k)=1 \quad \text { for all } \xi \in \mathbb{R}^4.
$$
Let $s \in \mathbb{R}$ and let $f \in H_x^s(\mathbb{R}^4)$. For every $k \in \mathbb{Z}^4$, we define the function $P_k f: \mathbb{R}^4 \rightarrow \mathbb{C}$ by
\begin{equation}\label{eq2.98}
 \left(P_k f\right)(x)=\mathcal{F}^{-1}(\psi(\xi-k) \hat{f}(\xi))(x) \quad \text { for } x \in \mathbb{R}^4.
\end{equation}
We exploit that these Fourier projections satisfy a unit-scale Bernstein inequality, namely for all $1 \leq r_1 \leq r_2 \leq \infty$ and for all $k \in \mathbb{Z}^4$ we have that
$$
\left\|P_k f\right\|_{L_x^{r_2}\left(\mathbb{R}^4\right)} \leq C\left(r_1, r_2\right)\left\|P_k f\right\|_{L_x^{r_1}\left(\mathbb{R}^4\right)}
$$
with a constant that is independent of $k \in \mathbb{Z}^4$.

We let $\left\{g_k\right\}_{k \in \mathbb{Z}^4}$ be a sequence of zero-mean, complex-valued Gaussian random variables on a probability space $(\Omega, \mathcal{A}, \mathbb{P})$. Given a complex-valued function $f \in H_x^s(\mathbb{R}^4)$ for some $s \in \mathbb{R}$, we define its randomization by
\begin{equation}\label{eq2.99}
  f^\omega:=\sum_{k \in \mathbb{Z}^4} g_k(\omega) P_k f.
\end{equation}
This quantity is understood as a Cauchy limit in $L_\omega^2(\Omega ; H_x^s(\mathbb{R}^4))$, and in the sequel, we will restrict ourselves to a subset $\Sigma \subset \Omega$ with $\mathbb{P}(\Sigma)=1$ such that $f^\omega \in H_x^s(\mathbb{R}^4)$ for every $\omega \in \Sigma$. It is worth mentioning that the randomization \eqref{eq2.99} almost surely does not regularize at the level of Sobolev spaces, see for instance \cite{{BDJL2019},{NBI2008}}. However, the free Schr\"odinger evolution $e^{i t H} f^\omega$ of the random data does enjoy various types of significantly improved space-time integrability properties, which crucially enter the proofs of our main results. This phenomenon is akin to the classical results of Paley and Zygmund \cite{REAC1930} on the improved integrability of random Fourier series.

Apart from the unit-scale frequency projections $P_k, k \in \mathbb{Z}^4$, defined in \eqref{eq2.98}, we will also make frequent use of the usual dyadic Littlewood-Paley projections $P_N, N \in 2^{\mathbb{Z}}$, which we introduce next. Let $\varphi \in C_c^{\infty}\left(\mathbb{R}^4\right)$ be a smooth bump function such that $\varphi(\xi)=1$ for $|\xi| \leq 1$ and $\varphi(\xi)=0$ for $|\xi|>2$. We define for every dyadic integer $N \in 2^{\mathbb{Z}}$,
$$
\widehat{P_N f}(\xi):=\left[\varphi(\frac{\xi}{N})-\varphi(\frac{2}{N})\right] \hat{f}(\xi).
$$
Moreover, for each dyadic integer $N \in 2^{\mathbb{Z}}$, we set
$$
\widehat{P_{\leq N} f}(\xi):=\varphi(\frac{\xi}{N}) \hat{f}(\xi), \quad \widehat{P_{>N} f}(\xi):=(1-\varphi(\frac{\xi}{N})) \hat{f}(\xi).
$$
We denote by $\widetilde{P}_N:=P_{\leq 8 N}-P_{\leq N / 8}$ fattened Littlewood-Paley projections with the property that $P_N=P_N \widetilde{P}_N$. Now, we recall the following Bernstein estimates for the dyadic Littlewood-Paley projections.

\begin{lemma}\label{L2.1}
Let $N \in 2^{\mathbb{Z}}$. For any $1 \leq r_1 \leq r_2 \leq \infty$ and any $s \geq 0$, it holds that
$$
\begin{aligned}
\left\|P_N f\right\|_{L_x^{r_2}\left(\mathbb{R}^4\right)} & \lesssim N^{\frac{4}{r_1}-\frac{4}{r_2}}\left\|P_N f\right\|_{L_x^{r_1}\left(\mathbb{R}^4\right)}, \\
\left\|P_{\leq N} f\right\|_{L_x^{r_2}\left(\mathbb{R}^4\right)} & \lesssim N^{\frac{4}{r_1}-\frac{4}{r_2}}\left\|P_{\leq N} f\right\|_{L_x^{r_1}\left(\mathbb{R}^4\right)}, \\
\left\||\nabla|^{ \pm s} P_N f\right\|_{L_x^{r_1}\left(\mathbb{R}^4\right)} & \sim N^{ \pm s}\left\|P_N f\right\|_{L_x^{r_1}\left(\mathbb{R}^4\right)}.
\end{aligned}
$$
\end{lemma}

\begin{lemma}\label{L2.2}(\cite{VDD2018}) Let $F(z)=|z|^k z$ with $k>0, s \geq 0$ and $1<p, p_1<\infty, 1<q_1 \leq \infty$ satisfying $\frac{1}{p}=\frac{1}{p_1}+\frac{k}{q_1}$. If $k$ is an even integer or $k$ is not an even integer with $[s] \leq k$, then there exists $C>0$ such that for all $u \in \mathscr{S}$,
$$
\|F(u)\|_{\dot{W}^{s, p}} \leq C\|u\|_{L^{q_1}}^k\|u\|_{\dot{W}^{s, p_1}} .
$$
A similar estimate holds with $\dot{W}^{s, p}, \dot{W}^{s, p_1}$-norms replaced by $W^{s, p}, W^{s, p_1}$ norms.
\end{lemma}
We let $\left\{\mathbf{e}_1, \mathbf{e}_2, \mathbf{e}_3, \mathbf{e}_4\right\}$ be an orthonormal basis of $\mathbb{R}^4$ and henceforth fix our coordinate system accordingly. To formulate certain local smoothing estimates for the Schr\"odinger evolution, we will use smooth frequency projections that localize the frequency variable in the direction of an element of the orthonormal basis $\left\{\mathbf{e}_1, \mathbf{e}_2, \mathbf{e}_3, \mathbf{e}_4\right\}$. To this end let $\phi \in C_c^{\infty}(\mathbb{R})$ be a smooth bump function supported around $\sim 1$. For every dyadic integer $N \in 2^{\mathbb{Z}}$ and for every $\ell=1, \ldots, 4$, we define
$$
\widehat{P_{N, \mathbf{e}_{\ell}} f}(\xi):=\phi\left(\frac{\left|\xi \cdot \mathbf{e}_{\ell}\right|}{ N}\right) \hat{f}(\xi).
$$
We may assume that the bump function $\phi$ is chosen so that for all dyadic integers $N \in 2^{\mathbb{Z}}$, the frequency projections satisfy
\begin{equation}\label{eq2.1}
 \left(1-P_{N, \mathbf{e}_1}\right)\left(1-P_{N, \mathbf{e}_2}\right)\left(1-P_{N, \mathbf{e}_3}\right)\left(1-P_{N, \mathbf{e}_4}\right) P_N=0.
\end{equation}

\begin{lemma}\label{L2.3}(see \cite{JWZY2024})
If  $V \in \mathcal{K} \cap L^{\frac{d}{2}}$ and $\left\|V_{-}\right\|_{\mathcal{K}}<d(d-2)\alpha(d)$, then
$$
\left\|\mathcal{L}^{\frac{s}{2}} f\right\|_{L^r} \sim\|f\|_{\dot{W}^{s, r}}, \quad\left\|(1+\mathcal{L})^{\frac{s}{2}} f\right\|_{L^r} \sim\|f\|_{W^{s, r}},
$$
where $1<r<\frac{d}{s}$ and $0 \leq s \leq 2$.
\end{lemma}
\textbf{Notation:}
\begin{itemize}
  \item $C>0$ denotes an absolute constant which only depends on fixed parameters and whose value may change from line to line.
  \item $X \lesssim Y$ indicates that $X \leq C Y$ and we use the notation $X \sim Y$ if $X \lesssim Y \lesssim X$.
  \item $X \lesssim \nu Y$, if the implicit constant depends on a parameter $\nu$.
  \item $X \ll Y$, if the implicit constant should be regarded as small.
  \item $\langle\nabla\rangle:=(1-\Delta)^{\frac{1}{2}}$, $\langle x\rangle:=\left(1+|x|^2\right)^{\frac{1}{2}}$ as well as $\langle N\rangle:=\left(1+N^2\right)^{\frac{1}{2}}$.
\end{itemize}

\section{Functional framework}
In this section we introduce the precise functional framework. We begin by recalling the usual Strichartz estimates for the Schr\"odinger propagator in four space dimensions. An exponent pair $(q, r)$ is called admissible if $2 \leq q, r \leq \infty$ and the following scaling condition is satisfied
$$
\frac{2}{q}+\frac{4}{r}=2.
$$
\begin{proposition}(see \cite{{MKTT1998},{YH2016}})\label{P3.1}
Let $I \subset \mathbb{R}$ be a time interval and let $(q, r),(\tilde{q}, \tilde{r})$ be admissible pairs. Then we have
\begin{equation}\label{eq3.1}
 \left\|e^{i t H} f\right\|_{L_t^q L_x^r\left(I \times \mathbb{R}^4\right)} \lesssim\|f\|_{L_x^2\left(\mathbb{R}^4\right)},
\end{equation}
\begin{equation}\label{eq3.2}
\left\|\int_I e^{-i s H} h(s, \cdot) d s\right\|_{L_x^2\left(\mathbb{R}^4\right)} \lesssim\|h\|_{L_t^{q^{\prime}} L_x^{r^{\prime}}\left(I \times \mathbb{R}^4\right)} .
\end{equation}
Assuming that $0 \in I$ we also have
\begin{equation}\label{eq3.3}
\left\|\int_0^t e^{i(t-s) H} h(s, \cdot) d s\right\|_{L_t^q L_x^r\left(I \times \mathbb{R}^4\right)} \lesssim\|h\|_{L_t^{\tilde{q}^{\prime}} L_x^{\tilde{r}^{\prime}}\left(I \times \mathbb{R}^4\right)}.
\end{equation}
\end{proposition}

Next we introduce the lateral spaces $L_{\mathbf{e}_{\ell}}^{p, q}$ and $W_{\mathbf{e}_{\ell}}^{p, q}$, where $\left\{\mathbf{e}_1, \mathbf{e}_2, \mathbf{e}_3, \mathbf{e}_4\right\}$ is a fixed orthonormal basis of $\mathbb{R}^4$. Let $I \subset \mathbb{R}$, we define the lateral spaces $L_{\mathbf{e}_{\ell}}^{p, q}\left(I \times \mathbb{R}^4\right)$$(\ell=1,2,3,4)$ with norms
$$
\|h\|_{L_{\mathbf{e}_{\ell}}^{p, q}\left(I \times \mathbb{R}^4\right)}:=\left(\int_{\mathbb{R}_{x_{\ell}}}\left(\int_I \int_{\mathbb{R}_{x^{\prime}}^3}\left|h\left(t, x_{\ell}, x^{\prime}\right)\right|^q d x^{\prime} d t\right)^{\frac{p}{q}} d x_{\ell}\right)^{\frac{1}{p}}
$$
and lateral spaces $W_{\mathbf{e}_{\ell}}^{p, q}\left(I \times \mathbb{R}^4\right)(\ell=1,2,3,4)$ with norms
$$
\|h\|_{W_{\mathbf{e}_{\ell}}^{p, q}\left(I \times \mathbb{R}^4\right)}:=\left(\int_{\mathbb{R}_{x_{\ell}}}\left(\int_I \int_{\mathbb{R}_{x^{\prime}}^3}\left[\left|\nabla h\left(t, x_{\ell}, x^{\prime}\right)\right|^q+\left|h\left(t, x_{\ell}, x^{\prime}\right)\right|^q\right] d x^{\prime} d t\right)^{\frac{p}{q}} d x_{\ell}\right)^{\frac{1}{p}}
$$
and the usual modifications when $p=\infty$ or $q=\infty$. The most important members of this family of spaces are the local smoothing space $L_{\mathbf{e}_{\ell}}^{\infty, 2}$ and the inhomogeneous local smoothing space $L_{\mathbf{e}_{\ell}}^{1,2}$, which allow us to gain derivatives.

\begin{lemma}\label{L3.1}
Let $I \subset \mathbb{R}$ be a time interval. Let $2 \leq p, q \leq \infty$ with $\frac{1}{p}+\frac{1}{q}=\frac{1}{2}, N \in 2^{\mathbb{Z}}$ any dyadic integer and $\ell \in\{1,2,3,4\}$. Then it holds that
\begin{equation}\label{eq3.4}
  \left\|e^{i t H} P_N f\right\|_{L_{\mathbf{e}_{\ell}}^{p, q}\left(I \times \mathbb{R}^4\right)}  \lesssim N^{\frac{4}{p}-\frac{1}{2}}\|f\|_{L_x^2\left(\mathbb{R}^4\right)}, \quad p \leq q,
\end{equation}
\begin{equation}\label{eq3.5}
\left\|e^{i t H} P_{N, \mathbf{e}_{\ell}} P_N f\right\|_{L_{\mathbf{e}_{\ell}}^{p, q}\left(I \times \mathbb{R}^4\right)}  \lesssim N^{\frac{4}{p}-\frac{1}{2}}\|f\|_{L_x^2\left(\mathbb{R}^4\right)}, \quad p \geq q.
\end{equation}
By duality we also have that
\begin{equation}\label{eq3.6}
 \left\|\int_I e^{-i s H} P_N h(s, \cdot) d s\right\|_{L_x^2\left(\mathbb{R}^4\right)} \lesssim  N^{\frac{4}{p}-\frac{1}{2}}\|h\|_{L_{\mathbf{e}_{\ell}}^{p^{\prime}, q^{\prime}}\left(I \times \mathbb{R}^4\right)}, \quad p \leq q,
\end{equation}
\begin{equation}\label{eq3.7}
  \left\|\int_I e^{-i s H} P_{N, \mathbf{e}_{\ell}} P_N h(s, \cdot) d s\right\|_{L_x^2\left(\mathbb{R}^4\right)} \lesssim  N^{\frac{4}{p}-\frac{1}{2}}\|h\|_{L_{\mathbf{e}_{\ell}}^{p^{\prime}, q^{\prime}}\left(I \times \mathbb{R}^4\right)}, \quad p \geq q.
\end{equation}
\end{lemma}
\begin{proof}
First, we prove \eqref{eq3.4}. On the one hand, the maximal function estimate from Lemma 4.1 in \cite{AICK2007} asserts that

$$
\left\|e^{i t H} P_N f\right\|_{L_{\mathbf{e}_{\ell}}^{2, \infty}\left(I \times \mathbb{R}^4\right)} \lesssim N^{\frac{3}{2}}\|f\|_{L_x^2\left(\mathbb{R}^4\right)}.
$$
On the other hand, by using Fubini's theorem, Bernstein estimates and \eqref{eq3.4}, it follows that
\begin{eqnarray*}
\left\|e^{i t H} P_N f\right\|_{L_{\mathbf{e}_{\ell}}^{4,4}\left(I \times \mathbb{R}^4\right)} & =&\left\|e^{i t H} P_N f\right\|_{L_t^4 L_x^4\left(I \times \mathbb{R}^4\right)} \lesssim N^{\frac{1}{2}}\left\|e^{i t H} P_N f\right\|_{L_t^4 L_x^{\frac{8}{3}}\left(I \times \mathbb{R}^4\right)} \\
& \lesssim& N^{\frac{1}{2}}\|f\|_{L_x^2\left(\mathbb{R}^4\right)}
\end{eqnarray*}
Hence, the estimate \eqref{eq3.4} follows by interpolation. Analogously, \eqref{eq3.5} is a consequence of the following local smoothing estimate from Lemma 4.2 in \cite{AICK2007} asserts that
$$
\left\|e^{i t H} P_N P_{N, \mathbf{e}_{\ell}} f\right\|_{L_{\mathbf{e}_{\ell}}^{\infty, 2}\left(I \times \mathbb{R}^4\right)} \lesssim N^{-\frac{1}{2}}\|f\|_{L_x^2\left(\mathbb{R}^4\right)}
$$
and interpolation.
\end{proof}

Note that the lateral space norms $\|h\|_{L_{\mathbf{e}_{\ell}}^{p, q}\left(I \times \mathbb{R}^4\right)}$ are continuous as functions of the endpoints of the time interval $I$ and for $p, q<\infty$ have the following time-divisibility property
\begin{equation}\label{eq3.99}
  \left\|\left\{\|h\|_{L_{\mathbf{e}_{\ell}}^{p, q}\left(I_j \times \mathbb{R}^4\right)}\right\}_{j=1}^J\right\|_{\ell_j^{\max \{p, q\}}} \leq\|h\|_{L_{\mathbf{e}_{\ell}}^{p, q}\left(I \times \mathbb{R}^4\right)}
\end{equation}
for any partition of a time interval $I$ into consecutive intervals $I_j, j=1, \ldots, J$, with disjoint interiors. The estimate \eqref{eq3.99} is a consequence of Minkowski's inequality and the embedding properties of the sequence spaces $\ell^r$. For $p, q<\infty$, \eqref{eq3.99} allows to partition the time interval $I$ into a controlled number of subintervals on each of which the restricted norm is arbitrarily small.

Finally, we are ready to give the precise definition of the space $X(I)$ to hold the solutions to the forced cubic NLS on a given time interval $I \subset \mathbb{R}$. It is built from dyadic pieces in the sense that
$$
\|v\|_{X(I)}:=\left(\sum_{N \in 2^\mathbb{Z}}\left\|P_N v\right\|_{X_N(I)}^2\right)^{\frac{1}{2}}.
$$
The dyadic subspace $X_N(I)$ scales at $\dot{H}_x^1\left(\mathbb{R}^4\right)$-regularity and consists of several Strichartz components and a maximal function type component. To provide its precise definition we introduce a fixed, sufficiently small, absolute constant $0<\varepsilon \ll 1$. Throughout this
work it will always be implicitly understood that $\varepsilon$ is chosen sufficiently small so that $\frac{1}{3}+3 \varepsilon \leq s$, where $\frac{1}{3}<s<1$ refers to the Sobolev regularity assumption for the random data in the statement of Theorem \ref{t1.1}. For every dyadic integer $N \in 2^{\mathbb{Z}}$ we then set
\begin{eqnarray*}
\left\|P_N v\right\|_{X_N(I)}&:=&N\left\|P_N v\right\|_{L_t^2 L_x^4\left(I \times \mathbb{R}^4\right)}+N\left\|P_N v\right\|_{L_t^3 L_x^3\left(I \times \mathbb{R}^4\right)}+N\left\|P_N v\right\|_{L_t^6 L_x^{\frac{12}{5}}\left(I \times \mathbb{R}^4\right)}\\
&&+\sum_{\ell=1}^4 N^{-\frac{1}{2}+\varepsilon}\left\|P_N v\right\|_{W_{\mathbf{e}_{\ell}}^{\frac{4}{2-\varepsilon}, \frac{4}{\varepsilon}}\left(I \times \mathbb{R}^4\right)}.
\end{eqnarray*}
We will estimate the forced cubic nonlinearity in the space $G(I)$ which is also built from dyadic pieces
$$
\|h\|_{G(I)}:=\left(\sum_{N \in 2^{\mathbb{Z}}}\left\|P_N h\right\|_{G_N(I)}^2\right)^{\frac{1}{2}}
$$
and whose dyadic subspaces are defined as
$$
\left\|P_N h\right\|_{G_N(I)}:=\inf _{P_N h=h_N^{(1)}+h_N^{(2)}}\left\{N\left\|h_N^{(1)}\right\|_{L_t^1 L_x^2\left(I \times \mathbb{R}^4\right)}+\sum_{\ell=1}^4 N^{\frac{1}{2}+\varepsilon}\left\|h_N^{(2)}\right\|_{L_{\mathbf{e}_{\ell}}^{\frac{4}{4-\varepsilon}, \frac{4}{2+\varepsilon}}\left(I \times \mathbb{R}^4\right)}\right\}.
$$
It will be convenient to introduce a space $Y(I)$, in which we will place the forcing term $F$. As usual, it is built from dyadic pieces in the sense that
$$
\|F\|_{Y(I)}:=\left(\sum_{N \in 2^\mathbb{Z}}\left\|P_N F\right\|_{Y_N(I)}^2\right)^{\frac{1}{2}},
$$
where we set
\begin{eqnarray*}
\left\|P_N F\right\|_{Y_N(I)}&=&\langle N\rangle^{\frac{1}{3}+3 \varepsilon}\left\|P_N F\right\|_{L_t^3 L_x^6\left(I \times \mathbb{R}^4\right)}+ \left\|P_N F\right\|_{L_t^6 L_x^6\left(I \times \mathbb{R}^4\right)}+\sum_{\ell=1}^4 N^{-\frac{1}{6}}\left\|P_N F\right\|_{W_{\mathbf{e}_{\ell}}^{\frac{4}{2-\varepsilon}, \frac{4}{\varepsilon}}(I \times \mathbb{R}^4)} \\
&&+\sum_{\ell=1}^4 \langle N\rangle^{\frac{1}{3}+3 \varepsilon}N^{\frac{1}{2}-\varepsilon}\left\|P_{N, \mathbf{e}_\ell}P_N F\right\|_{L_{\mathbf{e}_{\ell}}^{ \frac{4}{\varepsilon},\frac{4}{2-\varepsilon}}\left(I \times \mathbb{R}^4\right)}.
\end{eqnarray*}

In the next lemma we collect continuity and time-divisibility properties of the $X(I)$ and $Y(I)$ norms that we will repeatedly make use of.

\begin{proposition}(see \cite{{BDJL2019}})\label{P3.2}
(i) Let $I \subset \mathbb{R}$ be a closed interval. Assume that $\|v\|_{X(I)}<\infty$ and $\|F\|_{Y(I)}<\infty$. Then the mappings
$$
 t \mapsto\|v\|_{X([\inf I, t])}, \quad  t \mapsto\|F\|_{Y([\inf I, t])}
$$
and
$$
 t \mapsto\|v\|_{X([t, \sup I])}, \quad  t \mapsto\|F\|_{Y([t, \sup I])}
$$
are continuous with analogous statements for half-open and open intervals.

(ii) Let $I \subset \mathbb{R}$ be an interval and let $v \in X(I), F \in Y(I)$. For any partition of the interval $I$ into consecutive intervals $I_j, j=1, \ldots, J$, with disjoint interiors it holds that
$$
\left\|\left\{\|v\|_{X\left(I_j\right)}\right\}_{j=1}^J\right\|_{\ell_j^{\frac{4}{\varepsilon}}} \leq\|v\|_{X(I)}, \quad\left\|\left\{\|F\|_{Y\left(I_j\right)}\right\}_{j=1}^J\right\|_{\ell_j^{\frac{4}{\varepsilon}}} \leq\|F\|_{Y(I)}.
$$
\end{proposition}

\begin{lemma}\label{L3.2}
Let $J, I \subset \mathbb{R}$ be time intervals with $\sup J \leq \inf I$ and let $N \in 2^{\mathbb{Z}}$ be any dyadic integer. Then we have for any admissible Strichartz pair $(q, r)$ that
\begin{equation}\label{eq3.8}
  N\left\|\int_J e^{i(t-s) H} P_N u(s) d s\right\|_{L_t^q L_x^r\left(I \times \mathbb{R}^4\right)} \lesssim \sum_{\ell=1}^4 N^{\frac{1}{2}+\varepsilon}\left\|P_N u\right\|_{L_{\mathbf{e}_{\ell}}^{\frac{4}{4 -\varepsilon}, \frac{4}{2+\varepsilon}}\left(J \times \mathbb{R}^4\right)}.
\end{equation}
Furthermore, it holds that
\begin{equation}\label{eq3.9}
  \sum_{\ell=1}^4 N^{-\frac{1}{2}+\varepsilon}\left\|\int_J e^{i(t-s) H} P_N u(s) d s\right\|_{L_{\mathbf{e}_{\ell}}^{\frac{4}{4-\varepsilon}, \frac{4}{\varepsilon}}\left(I \times \mathbb{R}^4\right)} \lesssim \sum_{\ell=1}^4 N^{\frac{1}{2}+\varepsilon}\left\|P_N u\right\|_{L_{\mathbf{e}_{\ell}}^{\frac{4}{4-\varepsilon},\frac{4}{2+\varepsilon}}\left(J \times \mathbb{R}^4\right)}.
\end{equation}
\end{lemma}
\begin{proof}
The left-hand sides of \eqref{eq3.8} and \eqref{eq3.9} are both bounded by
\begin{equation}\label{eq3.10}
  N\left\|\int_J e^{-i s H} P_N u(s) d s\right\|_{L_x^2\left(\mathbb{R}^4\right)}
\end{equation}
by the Strichartz estimate \eqref{eq3.1} and by the estimate \eqref{eq3.4} for the lateral spaces, respectively. Relying on the identity \eqref{eq2.1}, we now further frequency decompose $P_N u$ into
\begin{eqnarray*}
P_N u&= & P_{N, \mathbf{e}_1} P_N u+P_{N, \mathbf{e}_2}\left(1-P_{N, \mathbf{e}_1}\right) P_N u+P_{N, \mathbf{e}_3}\left(1-P_{N, \mathbf{e}_2}\right)\left(1-P_{N, \mathbf{e}_1}\right) P_N u \\
&& +P_{N, \mathbf{e}_4}\left(1-P_{N, \mathbf{e}_3}\right)\left(1-P_{N, \mathbf{e}_2}\right)\left(1-P_{N, \mathbf{e}_1}\right) P_N u .
\end{eqnarray*}
Using the boundedness of the projections $\left(1-P_{N, \mathbf{e}_{\ell}}\right)$ on $L_x^2(\mathbb{R}^4)$, we can then estimate \eqref{eq3.10} by
$$
\sum_{\ell=1}^4 N\left\|\int_J e^{-i s H} P_{N, \mathbf{e}_{\ell}} P_N h(s) d s\right\|_{L_x^2\left(\mathbb{R}^4\right)}.
$$
Finally, by the dual estimate \eqref{eq3.7} for the lateral spaces we conclude the desired bound
\begin{equation*}
   \sum_{\ell=1}^4 N\left\|\int_J e^{-i s H} P_{N, \mathbf{e}_{\ell}} P_N h(s) d s\right\|_{L_x^2\left(\mathbb{R}^4\right)} \lesssim  \sum_{\ell=1}^4N^{\frac{1}{2}+\varepsilon}\|P_{N}u\|_{L_{\mathbf{e}_{\ell}}^{\frac{4}{4-\varepsilon},\frac{4}{2+\varepsilon}}\left(J \times \mathbb{R}^4\right)}.
\end{equation*}
\end{proof}

\begin{lemma}\label{L3.3}
Let $I \subset \mathbb{R}$ be a time interval with $0=\inf I$ and let $N \in 2^{\mathbb{Z}}$ be any dyadic integer. Then we have for any admissible Strichartz pair $(q, r)$ that
\begin{equation}\label{eq3.11}
      N\left\|\int_0^t e^{i (t-s) H} P_{N}u(s) d s\right\|_{L_t^qL_x^r(I\times\mathbb{R}^4)} \lesssim  \sum_{\ell=1}^4N^{\frac{1}{2}+\varepsilon}\|P_{N}u\|_{L_{\mathbf{e}_{\ell}}^{\frac{4}{4-\varepsilon},\frac{4}{2+\varepsilon}}\left(I \times \mathbb{R}^4\right)}.
\end{equation}
Furthermore, it holds that
\begin{equation}\label{eq3.12}
      \sum_{\ell=1}^4N^{\frac{1}{2}+\varepsilon}\left\|\int_0^t e^{i (t-s) H} P_{N}u(s) d s\right\|_{L_{\mathbf{e}_{\ell}}^{\frac{4}{2-\varepsilon},\frac{4}{\varepsilon}}\left(I \times \mathbb{R}^4\right)} \lesssim  \sum_{\ell=1}^4N^{\frac{1}{2}+\varepsilon}\|P_{N}u\|_{L_{\mathbf{e}_{\ell}}^{\frac{4}{4-\varepsilon},\frac{4}{2+\varepsilon}}\left(I \times \mathbb{R}^4\right)}.
\end{equation}
\end{lemma}
\begin{proof}
We only prove \eqref{eq3.12} in detail because the proof of \eqref{eq3.11} is similar. Without loss of generality, we may assume that
\begin{equation*}
  \sum_{\ell=1}^4N^{\frac{1}{2}+\varepsilon}\|P_{N}u\|_{L_{\mathbf{e}_{\ell}}^{\frac{4}{4-\varepsilon},\frac{4}{2+\varepsilon}}\left(I \times \mathbb{R}^4\right)}=1.
\end{equation*}
Now, it suffices to verify that
\begin{equation*}
         N^{\frac{1}{2}+\varepsilon}\left\|\int_0^t e^{i (t-s) H} P_{N}u(s) d s\right\|_{L_{\mathbf{e}_{1}}^{\frac{4}{2-\varepsilon},\frac{4}{\varepsilon}}\left(I \times \mathbb{R}^4\right)} \lesssim 1.
\end{equation*}
Using the time divisibility property \eqref{eq3.8} of the lateral spaces, we can proceed inductively to construct for every $n \in \mathbb{N}$ a partition $\left\{I_j^n\right\}_{j=1, \ldots, 2^n}$ of the interval $I$ into consecutive intervals with disjoint interiors such that for $j=1, \ldots, 2^n$,
\begin{equation}\label{eq3.13}
     \sum_{\ell=1}^4N^{\frac{1}{2}+\varepsilon}\|P_{N}u\|_{L_{\mathbf{e}_{\ell}}^{\frac{4}{4-\varepsilon},\frac{4}{2+\varepsilon}}\left(I_j^n \times \mathbb{R}^4\right)}\leq 2^{-(\frac{1}{2}+\frac{\varepsilon}{4})n}.
\end{equation}
We then perform a Whitney type decomposition of the interval $I$ and obtain that for almost every $t_1, t_2 \in I$ with $t_1<t_2$, there exist unique $n \in \mathbb{N}$ and $j \in\left\{1, \ldots, 2^n\right\}$ such that $t_1 \in I_j^n$ and $t_2 \in I_{j+1}^n$. Correspondingly, we may write
$$
\int_0^t e^{i(t-s) H} P_N h(s) d s=\sum_{n \in \mathbb{N}} \sum_{j=1}^{2^n} \chi_{I_{j+1}^n}(t) \int_{I_j^n} e^{i(t-s) H} P_N h(s) d s
$$
with the understanding that $I_{2^n+1}^n=\emptyset$ and where $\chi_{I_{j+1}^n}(t)$ denotes a sharp cut-off function to the interval $I_{j+1}^n$. To somewhat ease the notation in the following, we shall write $(p, q)=\left(\frac{4}{2-\varepsilon}, \frac{4}{\varepsilon}\right)$. Note that by Lemma \ref{L3.2} and by \eqref{eq3.13}, we have for any $n \in \mathbb{N}$ and $j \in\left\{1, \ldots, 2^n\right\}$ the bound
$$
\begin{aligned}
& N^{-\frac{1}{2}+\varepsilon}\left\|\int_{I_j^n} e^{i(t-s) H} P_N h(s) d s\right\|_{L_{\mathbf{e}_1}^{p, q}\left(I_{j+1}^n \times \mathbb{R}^4\right)} \lesssim   \sum_{\ell=1}^4N^{\frac{1}{2}+\varepsilon}\|P_{N}u\|_{L_{\mathbf{e}_{\ell}}^{\frac{4}{4-\varepsilon},\frac{4}{2+\varepsilon}}\left(I_j^n \times \mathbb{R}^4\right)}\leq 2^{-(\frac{1}{2}+\frac{\varepsilon}{4})n}.
\end{aligned}
$$
Hence, using also that $\frac{p}{q} \leq 1$, we compute that
\begin{eqnarray*}
&&N^{-\frac{1}{2}+\varepsilon}\left\|\int_0^t e^{i(t-s) H} P_N h(s) d s\right\|_{L_{\mathbf{e}_1}^{p, q}\left(I  \times \mathbb{R}^4\right)}\\
& \leq& N^{-\frac{1}{2}+\varepsilon} \sum_{n \in \mathbb{N}}\left\|\sum_{j=1}^{2^n} \chi_{I_{j+1}^n} \int_{I_j^n} e^{i(t-s) H} P_N h(s) d s\right\|_{L_{\mathbf{e}_1}^{p, q}\left(I  \times \mathbb{R}^4\right)} \\
& =&N^{-\frac{1}{2}+\varepsilon} \sum_{n \in \mathbb{N}}\left(\int_{\mathbb{R}_{x_1}}\left(\sum_{j=1}^{2^n} \int_{I_{j+1}^n \mathbb{R}_{x^{\prime}}^3}\left|\int_{I_j^n} e^{i(t-s) H} P_N h(s) d s\right|^q d x^{\prime} d t\right)^{\frac{p}{q}} d x_1\right)^{\frac{1}{p}} \\
& \leq& \sum_{n \in \mathbb{N}}\left(\sum_{j=1}^{2^n}\left(N^{-\frac{1}{2}+\varepsilon}\left\|\int_{I_j^n} e^{i(t-s) H} P_N h(s) d s\right\|_{L_{e_1}^{p, q}\left(I_{j+1}^n \times \mathbb{R}^4\right)}\right)^p\right)^{\frac{1}{p}} \\
& \lesssim& \sum_{n \in \mathbb{N}} 2^{\frac{n}{p}} 2^{-\left(\frac{1}{2}+\frac{\varepsilon}{4}\right) n} \\
& \simeq& \sum_{n \in \mathbb{N}} 2^{-\frac{\varepsilon}{2} n} \\
& \lesssim& 1,
\end{eqnarray*}
which implies that \eqref{eq3.13}  holds.
\end{proof}

Now, we give the following linear estimate.
\begin{lemma}\label{L3.4}
Let $I \subset \mathbb{R}$ be a time interval with $t_0 \in I$ and let $v_0 \in H_x^1(\mathbb{R}^4)$. Assume that $v: I \times \mathbb{R}^4 \rightarrow \mathbb{C}$ is a solution to
\begin{equation}\label{eq3.14}
  \left\{\begin{aligned}
\left(i \partial_t+\Delta-V\right) v & =h \text { on } I \times \mathbb{R}^4 \\
v\left(t_0\right) & =v_0.
\end{aligned}\right.
\end{equation}
Then we have for any dyadic integer $N \in 2^{\mathbb{Z}}$ that
\begin{equation}\label{eq3.15}
  N\left\|P_N v\right\|_{L_t^{\infty} L_x^2\left(I \times \mathbb{R}^4\right)}+\left\|P_N v\right\|_{X_N(I)} \lesssim N\left\|P_N v_0\right\|_{L_x^2\left(\mathbb{R}^4\right)}+\left\|P_N h\right\|_{G_N(I)}.
\end{equation}
Consequently, it holds that
\begin{equation}\label{eq3.16}
  \|v\|_{L_t^{\infty} H_x^1\left(I \times \mathbb{R}^4\right)}+\|v\|_{X(I)} \lesssim\left\|v_0\right\|_{H_x^1\left(\mathbb{R}^4\right)}+\|h\|_{G(I)}.
\end{equation}
\end{lemma}
\begin{proof}
Without loss of generality we may assume that $0=t_0=$ $\inf I$. By Duhamel's formula for the solution to the inhomogeneous Schr\"odinger equation \eqref{eq3.13}, we have for any dyadic integer $N \in 2^{\mathbb{Z}}$ that
\begin{equation}\label{eq3.17}
  P_N v(t)=e^{i t H} P_N v_0-i \int_0^t e^{i(t-s) H} P_N h(s) d s.
\end{equation}
By Proposition \ref{P3.1} and Lemma \ref{L3.1} it then holds that
\begin{eqnarray*}
&&N\left\|e^{i t H} P_N v_0\right\|_{L_t^{\infty} L_x^2\left(I \times \mathbb{R}^4\right)}+\left\|e^{i t H} P_N v_0\right\|_{X_N(I)}\\
&=&N\left\|e^{i t H} P_N v_0\right\|_{L_t^{\infty} L_x^2\left(I \times \mathbb{R}^4\right)}+N\left\|e^{i t H} P_N v_0\right\|_{L_t^2 L_x^4\left(I \times \mathbb{R}^4\right)}+N\left\|e^{i t H} P_N v_0\right\|_{L_t^3 L_x^3\left(I \times \mathbb{R}^4\right)}\\
&&+N\left\|e^{i t H} P_N v_0\right\|_{L_t^6 L_x^{\frac{12}{5}}\left(I \times \mathbb{R}^4\right)}\\
& \lesssim &N\left\|P_N v_0\right\|_{L_x^2\left(\mathbb{R}^4\right)}.
\end{eqnarray*}
In order to complete the proof of \eqref{eq3.15}, it remains to verify that the Duhamel term in \eqref{eq3.17} satisfies for any admissible Strichartz pair $(q, r)$ that
\begin{eqnarray}\label{eq3.18}
&& N\left\|\int_0^t e^{i(t-s) H} P_N h(s) d s\right\|_{L_t^q L_x^r\left(I \times \mathbb{R}^4\right)} +\sum_{\ell=1}^4 N^{-\frac{1}{2}+\varepsilon}\left\|\int_0^t e^{i(t-s) H} P_N h(s) d s\right\|_{L_{\mathbf{e}_{\ell}}^{\frac{4}{2-\varepsilon}, \frac{4}{\varepsilon}}\left(I \times \mathbb{R}^4\right)}\nonumber\\
  &\lesssim& N\left\|P_N h\right\|_{L_t^1 L_x^2\left(I \times \mathbb{R}^4\right)}
\end{eqnarray}
as well as
\begin{eqnarray}\label{eq3.19}
&& N\left\|\int_0^t e^{i(t-s) H} P_N h(s) d s\right\|_{L_t^q L_x^r\left(I \times \mathbb{R}^4\right)} +\sum_{\ell=1}^4 N^{-\frac{1}{2}+\varepsilon}\left\|\int_0^t e^{i(t-s) H} P_N h(s) d s\right\|_{L_{\mathbf{e}_{\ell}}^{\frac{4}{2-\varepsilon}, \frac{4}{\varepsilon}}\left(I \times \mathbb{R}^4\right)}\nonumber\\
  &\lesssim& N^{\frac{1}{2}+\varepsilon}\left\|P_N h\right\|_{L_t^1 L_x^2\left(I \times \mathbb{R}^4\right)}.
\end{eqnarray}
The proof of \eqref{eq3.18} is standard and therefore omitted, while the estimate \eqref{eq3.19} is provided by Lemma \ref{L3.3}.
\end{proof}

\section{The estimates of nonlinearity}
\begin{proposition}\label{P4.1}
Let $N_1 \gtrsim N$ and $N_1 \geq N_2 \geq N_3$ be dyadic integers. Let $\mathbf{e} \in\left\{\mathbf{e}_1, \ldots, \mathbf{e}_4\right\}$ and let $I \subset \mathbb{R}$ be a time interval. Then the following trilinear estimates hold where all space-time norms are taken over $I \times \mathbb{R}^4$.
\begin{eqnarray}\label{eq4.1}
  &&N\left\|P_N\left(P_{N_1} v_1 P_{N_2} v_2 P_{N_3} v_3\right)\right\|_{L_t^1 W_x^{1,2}}\\
  &\lesssim&N\left[\frac{1}{N_1}\left(\frac{N_3}{N_2}\right)^{\frac{2}{3}}+\left(\frac{N_3}{N_2}\right)^{\frac{1}{3}}+\left(\frac{N_3}{N_1}\right)^{\frac{1}{3}}+\left(\frac{N_2}{N_1}\right)^{\frac{1}{3}} \right]\left\|P_{N_1} v_1\right\|_{X_{N_1}}\left\|P_{N_2} v_2\right\|_{X_{N_2}}\left\|P_{N_3} v_3\right\|_{X_{N_3}}, \nonumber
\end{eqnarray}
\begin{eqnarray}\label{eq4.2}
  &&N\left\|P_N\left(P_{N_1} v_1 P_{N_2} F_2 P_{N_3} v_3\right)\right\|_{L_t^1 W_x^{1,2}}\\
  &\lesssim&N\left[\frac{1}{N_1}\left(\frac{N_3}{N_2}\right)^{\frac{1}{3}}+\left(\frac{N_3}{N_2}\right)^{\frac{1}{3}}+\left(\frac{N_2}{N_1}\right)^{\frac{2}{3}}+\left(\frac{N_2}{N_1}\right)^{\frac{1}{3}} \right]\left\|P_{N_1} v_1\right\|_{X_{N_1}}\left\|P_{N_2} F_2\right\|_{Y_{N_2}}\left\|P_{N_3} v_3\right\|_{X_{N_3}}, \nonumber
\end{eqnarray}
\begin{eqnarray}\label{eq4.3}
  &&N\left\|P_N\left(P_{N_1} v_1 P_{N_2} v_2 P_{N_3} F_3\right)\right\|_{L_t^1 W_x^{1,2}}\\
  &\lesssim&N\left[\frac{1}{N_1}\left(\frac{N_3}{N_2}\right)^{\frac{2}{3}}+\left(\frac{N_3}{N_2}\right)^{\frac{2}{3}}+\left(\frac{N_2}{N_1}\right)^{\frac{2}{3}}+ \frac{N_3}{N_2}  \right]\left\|P_{N_1} v_1\right\|_{X_{N_1}}\left\|P_{N_2} v_2\right\|_{X_{N_2}}\left\|P_{N_3} F_3\right\|_{Y_{N_3}}, \nonumber
\end{eqnarray}
\begin{eqnarray}\label{eq4.4}
  &&N\left\|P_N\left(P_{N_1} v_1 P_{N_2} F_2 P_{N_3} F_3\right)\right\|_{L_t^1 W_x^{1,2}}\\
  &\lesssim&N\left[\frac{1}{N_1}\left(\frac{N_3}{N_2}\right)^{\frac{1}{3}}+\left(\frac{N_3}{N_2}\right)^{\frac{1}{3}}+\left(\frac{N_2}{N_1}\right)^{\frac{2}{3}}+ \frac{N_3}{N_2}  \right]\left\|P_{N_1} v_1\right\|_{X_{N_1}}\left\|P_{N_2} F_2\right\|_{Y_{N_2}}\left\|P_{N_3} F_3\right\|_{Y_{N_3}}, \nonumber
\end{eqnarray}
\begin{eqnarray}\label{eq4.5}
&&N^{\frac{1}{2}+\varepsilon}\left\|P_N\left(P_{N_1} F_1 P_{N_2} F_2 P_{N_3} F_3\right)\right\|_{W_{\mathbf{e}}^{\frac{4}{4-\varepsilon},\frac{4}{2+\varepsilon}}}\nonumber\\
&\lesssim&\left(\frac{N}{N_1}\right)^{\frac{1}{2}+\varepsilon}\left(\frac{N_3}{N_1}\right)^{\frac{1}{6}}\left\|P_{N_1} F_1\right\|_{Y_{N_1}}\left\|P_{N_2} F_2\right\|_{Y_{N_2}}\left\|P_{N_3} F_3\right\|_{Y_{N_3}},
\end{eqnarray}
\begin{eqnarray}\label{eq4.6}
&&N^{\frac{1}{2}+\varepsilon}\left\|P_N\left(P_{N_1} F_1 P_{N_2} v_2 P_{N_3} v_3\right)\right\|_{W_{\mathbf{e}}^{\frac{4}{4-\varepsilon},\frac{4}{2+\varepsilon}}}\nonumber\\
&\lesssim&\left(\frac{N}{N_2}\right)^{\frac{1}{2}+\varepsilon}\left(\frac{N_3}{N_2}\right)^{\frac{1}{2}-\varepsilon}\left\|P_{N_1} F_1\right\|_{Y_{N_1}}\left\|P_{N_2} v_2\right\|_{X_{N_2}}\left\|P_{N_3} v_3\right\|_{X_{N_3}},
\end{eqnarray}
\begin{eqnarray}\label{eq4.7}
&&N^{\frac{1}{2}+\varepsilon}\left\|P_N\left(P_{N_1} F_1 P_{N_2} F_2 P_{N_3} v_3\right)\right\|_{W_{\mathbf{e}}^{\frac{4}{4-\varepsilon},\frac{4}{2+\varepsilon}}}\nonumber\\
&\lesssim&\left(\frac{N}{N_3}\right)^{\frac{1}{2}+\varepsilon}\left(\frac{N_2}{N_1}\right)^{\frac{1}{6}} \cdot\left\|P_{N_1} F_1\right\|_{Y_{N_1}}\left\|P_{N_2} F_2\right\|_{Y_{N_2}}\left\|P_{N_3} v_3\right\|_{X_{N_3}},
\end{eqnarray}
\begin{eqnarray}\label{eq4.8}
&&N^{\frac{1}{2}+\varepsilon}\left\|P_N\left(P_{N_1} F_1 P_{N_2} v_2 P_{N_3} F_3\right)\right\|_{W_{\mathbf{e}}^{\frac{4}{4-\varepsilon},\frac{4}{2+\varepsilon}}}\nonumber\\
&\lesssim&\left(\frac{N}{N_2}\right)^{\frac{1}{2}+\varepsilon}\left(\frac{N_3}{N_2}\right)^{\frac{1}{6} }\left\|P_{N_1} F_1\right\|_{Y_{N_1}}\left\|P_{N_2} v_2\right\|_{X_{N_2}}\left\|P_{N_3} F_3\right\|_{Y_{N_3}}.
\end{eqnarray}
\end{proposition}
\begin{proof}
In the following all space-time norms are taken over $I \times \mathbb{R}^4$. We begin with the derivation of the estimates \eqref{eq4.1}-\eqref{eq4.4} where a deterministic solution $v$ appears at highest frequency. The proofs are simple applications of H\"older inequality and Bernstein estimates.
\textbf{Case 1:} $\mathbf{v_1v_2v_3}$
\begin{eqnarray*}
&& N\left\|P_N\left(P_{N_1} v_1 P_{N_2} v_2 P_{N_3} v_3\right)\right\|_{L_t^1W_x^{1,2}} \\
&\lesssim& N\left\|P_N\left(P_{N_1} v_1 P_{N_2} v_2 P_{N_3} v_3\right)\right\|_{L_t^1L_x^{2}}+N\left\|\nabla(P_N\left(P_{N_1} v_1 P_{N_2} v_2 P_{N_3} v_3\right))\right\|_{L_t^1L_x^{2}}\\
& \lesssim& N\left[\left\|\nabla P_{N_1} v_1\right\|_{L_t^2 L_x^4}\left\|P_{N_2} v_2\right\|_{L_t^3 L_x^6}\left\|P_{N_3} v_3\right\|_{L_t^6 L_x^{12}}+\left\|P_{N_1} v_1\right\|_{L_t^2 L_x^4}\left\|P_{N_2} v_2\right\|_{L_t^3 L_x^4}\left\|P_{N_3} v_3\right\|_{L_t^6 L_x^{\infty}}\right. \\
&&\left.+\left\| P_{N_1} v_1\right\|_{L_t^2 L_x^{12}}\left\|\nabla P_{N_2} v_2\right\|_{L_t^3 L_x^3}\left\|P_{N_3} v_3\right\|_{L_t^6 L_x^{12}}+\left\|  P_{N_1} v_1\right\|_{L_t^2 L_x^{12}}\left\|P_{N_2} v_2\right\|_{L_t^3 L_x^{\infty}}\left\|\nabla P_{N_3} v_3\right\|_{L_t^6 L_x^{\frac{12}{5}}}\right]\\
& \lesssim& N\left[N_1\left\|  P_{N_1} v_1\right\|_{L_t^2 L_x^4}\left\|P_{N_2} v_2\right\|_{L_t^3 L_x^6}\left\|P_{N_3} v_3\right\|_{L_t^6 L_x^{12}}+\left\|P_{N_1} v_1\right\|_{L_t^2 L_x^4}\left\|P_{N_2} v_2\right\|_{L_t^3 L_x^4}\left\|P_{N_3} v_3\right\|_{L_t^6 L_x^{\infty}}\right. \\
&&\left.+N_2\left\| P_{N_1} v_1\right\|_{L_t^2 L_x^{12}}\left\| P_{N_2} v_2\right\|_{L_t^3 L_x^3}\left\|P_{N_3} v_3\right\|_{L_t^6 L_x^{12}}+N_3\left\|  P_{N_1} v_1\right\|_{L_t^2 L_x^{12}}\left\|P_{N_2} v_2\right\|_{L_t^3 L_x^{\infty}}\left\|  P_{N_3} v_3\right\|_{L_t^6 L_x^{\frac{12}{5}}}\right]\\
& \lesssim& N\left\|P_{N_1} v_1\right\|_{L_t^2 L_x^4} N_2^{\frac{1}{3}}\left\|P_{N_2} v_2\right\|_{L_t^3 L_x^3} N_3^{\frac{5}{3}}\left\|P_{N_3} v_3\right\|_{L_t^6L_x^{\frac{12}{5}}}\\
&&+ N\cdot N_1\left\|P_{N_1} v_1\right\|_{L_t^2 L_x^4} N_2^{\frac{2}{3}}\left\|P_{N_2} v_2\right\|_{L_t^3 L_x^3} N_3^{\frac{4}{3}}\left\|P_{N_3} v_3\right\|_{L_t^6L_x^{\frac{12}{5}}} \\
&&+ N\cdot N_1^{\frac{2}{3}}\left\|P_{N_1} v_1\right\|_{L_t^2 L_x^4} N_2 \left\|P_{N_2} v_2\right\|_{L_t^3 L_x^3} N_3^{\frac{4}{3}}\left\|P_{N_3} v_3\right\|_{L_t^6L_x^{\frac{12}{5}}} \\
&&+ N\cdot N_1^{\frac{2}{3}}\left\|P_{N_1} v_1\right\|_{L_t^2 L_x^4} N_2^{\frac{4}{3}}\left\|P_{N_2} v_2\right\|_{L_t^3 L_x^3} N_3 \left\|P_{N_3} v_3\right\|_{L_t^6L_x^{\frac{12}{5}}} \\
& \simeq&\left(\frac{N}{N_1}\right)\left(\frac{N_3}{N_2}\right)^{\frac{2}{3}} N_1\left\|P_{N_1} v_1\right\|_{L_t^2 L_x^4} N_2\left\|P_{N_2} v_2\right\|_{L_t^3 L_x^3} N_3\left\|P_{N_3} v_3\right\|_{L_t^6 L_x^{\frac{12}{5}}}\\
&&+N\left(\frac{N_3}{N_2}\right)^{\frac{1}{3}} N_1\left\|P_{N_1} v_1\right\|_{L_t^2 L_x^4} N_2\left\|P_{N_2} v_2\right\|_{L_t^3 L_x^3} N_3\left\|P_{N_3} v_3\right\|_{L_t^6 L_x^{\frac{12}{5}}}\\
&&+N\left(\frac{N_3}{N_1}\right)^{\frac{1}{3}} N_1\left\|P_{N_1} v_1\right\|_{L_t^2 L_x^4} N_2\left\|P_{N_2} v_2\right\|_{L_t^3 L_x^3} N_3\left\|P_{N_3} v_3\right\|_{L_t^6 L_x^{\frac{12}{5}}}\\
&&+N\left(\frac{N_2}{N_1}\right)^{\frac{1}{3}} N_1\left\|P_{N_1} v_1\right\|_{L_t^2 L_x^4} N_2\left\|P_{N_2} v_2\right\|_{L_t^3 L_x^3} N_3\left\|P_{N_3} v_3\right\|_{L_t^6 L_x^{\frac{12}{5}}}\\
&\lesssim&N\left[\frac{1}{N_1}\left(\frac{N_3}{N_2}\right)^{\frac{2}{3}}+\left(\frac{N_3}{N_2}\right)^{\frac{1}{3}}+\left(\frac{N_3}{N_1}\right)^{\frac{1}{3}}+\left(\frac{N_2}{N_1}\right)^{\frac{1}{3}} \right]\left\|P_{N_1} v_1\right\|_{X_{N_1}}\left\|P_{N_2} v_2\right\|_{X_{N_2}}\left\|P_{N_3} v_3\right\|_{X_{N_3}}.
\end{eqnarray*}
\textbf{Case 2:} $\mathbf{v_1F_2v_3}$
\begin{eqnarray*}
&& N\left\|P_N\left(P_{N_1} v_1 P_{N_2} F_2 P_{N_3} v_3\right)\right\|_{L_t^1W_x^{1,2}} \\
&\lesssim& N\left\|P_N\left(P_{N_1} v_1 P_{N_2} F_2 P_{N_3} v_3\right)\right\|_{L_t^1L_x^{2}}+N\left\|\nabla(P_N\left(P_{N_1} v_1 P_{N_2} F_2 P_{N_3} v_3\right))\right\|_{L_t^1L_x^{2}}\\
& \lesssim& N\left[\left\|\nabla P_{N_1} v_1\right\|_{L_t^2 L_x^4}\left\|P_{N_2} F_2\right\|_{L_t^3 L_x^6}\left\|P_{N_3} v_3\right\|_{L_t^6 L_x^{12}}+\left\|P_{N_1} v_1\right\|_{L_t^2 L_x^4}\left\|P_{N_2} F_2\right\|_{L_t^3 L_x^6}\left\|P_{N_3} v_3\right\|_{L_t^6 L_x^{12}}\right. \\
&&\left.+\left\| P_{N_1} v_1\right\|_{L_t^2 L_x^{6}}\left\|\nabla P_{N_2} F_2\right\|_{L_t^3 L_x^6}\left\|P_{N_3} v_3\right\|_{L_t^6 L_x^{6}}+\left\|  P_{N_1} v_1\right\|_{L_t^2 L_x^{12}}\left\|P_{N_2} F_2\right\|_{L_t^3 L_x^{\infty}}\left\|\nabla P_{N_3} v_3\right\|_{L_t^6 L_x^{\frac{12}{5}}}\right]\\
& \lesssim& N\left[N_1\left\|  P_{N_1} v_1\right\|_{L_t^2 L_x^4}\left\|P_{N_2} F_2\right\|_{L_t^3 L_x^6}\left\|P_{N_3} v_3\right\|_{L_t^6 L_x^{12}}+\left\|P_{N_1} v_1\right\|_{L_t^2 L_x^4}\left\|P_{N_2} F_2\right\|_{L_t^3 L_x^6}\left\|P_{N_3} v_3\right\|_{L_t^6 L_x^{12}}\right. \\
&&\left.+N_2\left\| P_{N_1} v_1\right\|_{L_t^2 L_x^{6}}\left\| P_{N_2} F_2\right\|_{L_t^3 L_x^6}\left\|P_{N_3} v_3\right\|_{L_t^6 L_x^{6}}+N_3\left\|  P_{N_1} v_1\right\|_{L_t^2 L_x^{12}}\left\|P_{N_2} F_2\right\|_{L_t^3 L_x^{\infty}}\left\|  P_{N_3} v_3\right\|_{L_t^6 L_x^{\frac{12}{5}}}\right]\\
& \lesssim& N\left\|P_{N_1} v_1\right\|_{L_t^2 L_x^4} \left\|P_{N_2} F_2\right\|_{L_t^3 L_x^6} N_3^{\frac{4}{3}}\left\|P_{N_3} v_3\right\|_{L_t^6L_x^{\frac{12}{5}}}\\
&&+ N\cdot N_1\left\|P_{N_1} v_1\right\|_{L_t^2 L_x^4} \left\|P_{N_2} F_2\right\|_{L_t^3 L_x^6} N_3^{\frac{4}{3}}\left\|P_{N_3} v_3\right\|_{L_t^6L_x^{\frac{12}{5}}} \\
&&+ N\cdot N_1^{\frac{1}{3}}\left\|P_{N_1} v_1\right\|_{L_t^2 L_x^4} N_2 \left\|P_{N_2} F_2\right\|_{L_t^3 L_x^6} N_3 \left\|P_{N_3} v_3\right\|_{L_t^6L_x^{\frac{12}{5}}} \\
&&+ N\cdot N_1^{\frac{2}{3}}\left\|P_{N_1} v_1\right\|_{L_t^2 L_x^4} N_2^{\frac{2}{3}}\left\|P_{N_2} F_2\right\|_{L_t^3 L_x^6} N_3 \left\|P_{N_3} v_3\right\|_{L_t^6L_x^{\frac{12}{5}}} \\
& \simeq&\left(\frac{N}{N_1}\right)\left(\frac{N_3}{N_2}\right)^{\frac{1}{3}} N_1\left\|P_{N_1} v_1\right\|_{L_t^2 L_x^4} N_2^{\frac{1}{3}}\left\|P_{N_2} F_2\right\|_{L_t^3 L_x^6} N_3\left\|P_{N_3} v_3\right\|_{L_t^6 L_x^{\frac{12}{5}}}\\
&&+N\left(\frac{N_3}{N_2}\right)^{\frac{1}{3}} N_1\left\|P_{N_1} v_1\right\|_{L_t^2 L_x^4} N_2^{\frac{1}{3}}\left\|P_{N_2} F_2\right\|_{L_t^3 L_x^6} N_3\left\|P_{N_3} v_3\right\|_{L_t^6 L_x^{\frac{12}{5}}}\\
&&+N\left(\frac{N_2}{N_1}\right)^{\frac{2}{3}} N_1\left\|P_{N_1} v_1\right\|_{L_t^2 L_x^4} N_2^{\frac{1}{3}}\left\|P_{N_2} F_2\right\|_{L_t^3 L_x^6} N_3\left\|P_{N_3} v_3\right\|_{L_t^6 L_x^{\frac{12}{5}}}\\
&&+N\left(\frac{N_2}{N_1}\right)^{\frac{1}{3}} N_1\left\|P_{N_1} v_1\right\|_{L_t^2 L_x^4} N_2^{\frac{1}{3}}\left\|P_{N_2} F_2\right\|_{L_t^3 L_x^6} N_3\left\|P_{N_3} v_3\right\|_{L_t^6 L_x^{\frac{12}{5}}}\\
&\lesssim&N\left[\frac{1}{N_1}\left(\frac{N_3}{N_2}\right)^{\frac{1}{3}}+\left(\frac{N_3}{N_2}\right)^{\frac{1}{3}}+\left(\frac{N_2}{N_1}\right)^{\frac{2}{3}}+\left(\frac{N_2}{N_1}\right)^{\frac{1}{3}} \right]\left\|P_{N_1} v_1\right\|_{X_{N_1}}\left\|P_{N_2} F_2\right\|_{Y_{N_2}}\left\|P_{N_3} v_3\right\|_{X_{N_3}}.
\end{eqnarray*}
\textbf{Case 3:} $\mathbf{v_1v_2F_3}$
\begin{eqnarray*}
&& N\left\|P_N\left(P_{N_1} v_1 P_{N_2} v_2 P_{N_3} F_3\right)\right\|_{L_t^1W_x^{1,2}} \\
&\lesssim& N\left\|P_N\left(P_{N_1} v_1 P_{N_2} v_2 P_{N_3} F_3\right)\right\|_{L_t^1L_x^{2}}+N\left\|\nabla(P_N\left(P_{N_1} v_1 P_{N_2} v_2 P_{N_3} F_3\right))\right\|_{L_t^1L_x^{2}}\\
& \lesssim& N\left[\left\|\nabla P_{N_1} v_1\right\|_{L_t^2 L_x^4}\left\|P_{N_2} v_2\right\|_{L_t^3 L_x^4}\left\|P_{N_3} F_3\right\|_{L_t^6 L_x^{\infty}}+\left\|P_{N_1} v_1\right\|_{L_t^2 L_x^{4}}\left\|P_{N_2} v_2\right\|_{L_t^3 L_x^4}\left\|P_{N_3} F_3\right\|_{L_t^6 L_x^{\infty}}\right. \\
&&\left.+\left\| P_{N_1} v_1\right\|_{L_t^2 L_x^{6}}\left\|\nabla P_{N_2} v_2\right\|_{L_t^3 L_x^6}\left\|P_{N_3} F_3\right\|_{L_t^6 L_x^{6}}+\left\|  P_{N_1} v_1\right\|_{L_t^2 L_x^{6}}\left\|P_{N_2} v_2\right\|_{L_t^3 L_x^{6}}\left\|\nabla P_{N_3} F_3\right\|_{L_t^6 L_x^{6}}\right]\\
& \lesssim& N\left[N_1\left\|  P_{N_1} v_1\right\|_{L_t^2 L_x^4}\left\|P_{N_2} v_2\right\|_{L_t^3 L_x^4}\left\|P_{N_3} F_3\right\|_{L_t^6 L_x^{\infty}}+\left\|P_{N_1} v_1\right\|_{L_t^2 L_x^4}\left\|P_{N_2} v_2\right\|_{L_t^3 L_x^4}\left\|P_{N_3} F_3\right\|_{L_t^6 L_x^{\infty}}\right. \\
&&\left.+N_2\left\| P_{N_1} v_1\right\|_{L_t^2 L_x^{6}}\left\| P_{N_2} v_2\right\|_{L_t^3 L_x^6}\left\|P_{N_3} F_3\right\|_{L_t^6 L_x^{6}}+N_3\left\|  P_{N_1} v_1\right\|_{L_t^2 L_x^{6}}\left\|P_{N_2} v_2\right\|_{L_t^3 L_x^{6}}\left\|  P_{N_3} F_3\right\|_{L_t^6 L_x^{6}}\right]\\
& \lesssim& N\left\|P_{N_1} v_1\right\|_{L_t^2 L_x^4}N_2^{\frac{1}{3}}\left\|P_{N_2} v_2\right\|_{L_t^3 L_x^3}  \cdot N_3^{\frac{2}{3}}\left\|P_{N_3} F_3\right\|_{L_t^6L_x^{6}}\\
&&+ N\cdot N_1\left\|P_{N_1} v_1\right\|_{L_t^2 L_x^4}N_2^{\frac{1}{3}}\left\|P_{N_2} v_2\right\|_{L_t^3 L_x^3}  \cdot N_3^{\frac{2}{3}}\left\|P_{N_3} F_3\right\|_{L_t^6L_x^{6}} \\
&&+ N\cdot N_1^{\frac{1}{3}}\left\|P_{N_1} v_1\right\|_{L_t^2 L_x^4} N_2^{\frac{5}{3}} \left\|P_{N_2} v_2\right\|_{L_t^3 L_x^3} \left\|P_{N_3} F_3\right\|_{L_t^6L_x^{6}} \\
&&+ N\cdot N_1^{\frac{1}{3}}\left\|P_{N_1} v_1\right\|_{L_t^2 L_x^4} N_2^{\frac{2}{3}}\left\|P_{N_2} v_2\right\|_{L_t^3 L_x^3} N_3 \left\|P_{N_3} F_3\right\|_{L_t^6L_x^{6}} \\
& \simeq&\left(\frac{N}{N_1}\right)\left(\frac{N_3}{N_2}\right)^{\frac{2}{3}} N_1\left\|P_{N_1} v_1\right\|_{L_t^2 L_x^4} N_2\left\|P_{N_2} v_2\right\|_{L_t^3 L_x^3} \left\|P_{N_3} F_3\right\|_{L_t^6 L_x^{6}}\\
&&+N\left(\frac{N_3}{N_2}\right)^{\frac{2}{3}} N_1\left\|P_{N_1} v_1\right\|_{L_t^2 L_x^4} N_2\left\|P_{N_2} v_2\right\|_{L_t^3 L_x^3} \left\|P_{N_3} F_3\right\|_{L_t^6 L_x^{6}}\\
&&+N\left(\frac{N_2}{N_1}\right)^{\frac{2}{3}} N_1\left\|P_{N_1} v_1\right\|_{L_t^2 L_x^4} N_2\left\|P_{N_2} v_2\right\|_{L_t^3 L_x^3} N_3\left\|P_{N_3} F_3\right\|_{L_t^6 L_x^{6}}\\
&&+N\left(\frac{N_3}{N_2}\right)  N_1\left\|P_{N_1} v_1\right\|_{L_t^2 L_x^4} N_2\left\|P_{N_2} v_2\right\|_{L_t^3 L_x^3}  \left\|P_{N_3} F_3\right\|_{L_t^6 L_x^{6}}\\
&\lesssim&N\left[\frac{1}{N_1}\left(\frac{N_3}{N_2}\right)^{\frac{2}{3}}+\left(\frac{N_3}{N_2}\right)^{\frac{2}{3}}+\left(\frac{N_2}{N_1}\right)^{\frac{2}{3}}+ \frac{N_3}{N_2}  \right]\left\|P_{N_1} v_1\right\|_{X_{N_1}}\left\|P_{N_2} v_2\right\|_{X_{N_2}}\left\|P_{N_3} F_3\right\|_{Y_{N_3}}.
\end{eqnarray*}
\textbf{Case 4:} $\mathbf{v_1F_2F_3}$
\begin{eqnarray*}
&& N\left\|P_N\left(P_{N_1} v_1 P_{N_2} F_2 P_{N_3} F_3\right)\right\|_{L_t^1W_x^{1,2}} \\
&\lesssim& N\left\|P_N\left(P_{N_1} v_1 P_{N_2} F_2 P_{N_3} F_3\right)\right\|_{L_t^1L_x^{2}}+N\left\|\nabla(P_N\left(P_{N_1} v_1 P_{N_2} F_2 P_{N_3} F_3\right))\right\|_{L_t^1L_x^{2}}\\
& \lesssim& N\left[\left\|\nabla P_{N_1} v_1\right\|_{L_t^2 L_x^4}\left\|P_{N_2} F_2\right\|_{L_t^3 L_x^6}\left\|P_{N_3} F_3\right\|_{L_t^6 L_x^{12}}+\left\|P_{N_1} v_1\right\|_{L_t^2 L_x^{4}}\left\|P_{N_2} F_2\right\|_{L_t^3 L_x^6}\left\|P_{N_3} F_3\right\|_{L_t^6 L_x^{12}}\right. \\
&&\left.+\left\| P_{N_1} v_1\right\|_{L_t^2 L_x^{6}}\left\|\nabla P_{N_2} F_2\right\|_{L_t^3 L_x^6}\left\|P_{N_3} F_3\right\|_{L_t^6 L_x^{6}}+\left\|  P_{N_1} v_1\right\|_{L_t^2 L_x^{6}}\left\|P_{N_2} F_2\right\|_{L_t^3 L_x^{6}}\left\|\nabla P_{N_3} F_3\right\|_{L_t^6 L_x^{6}}\right]\\
& \lesssim& N\left[N_1\left\|  P_{N_1} v_1\right\|_{L_t^2 L_x^4}\left\|P_{N_2} F_2\right\|_{L_t^3 L_x^6}\left\|P_{N_3} F_3\right\|_{L_t^6 L_x^{12}}+\left\|P_{N_1} v_1\right\|_{L_t^2 L_x^4}\left\|P_{N_2} F_2\right\|_{L_t^3 L_x^6}\left\|P_{N_3} F_3\right\|_{L_t^6 L_x^{12}}\right. \\
&&\left.+N_2\left\| P_{N_1} v_1\right\|_{L_t^2 L_x^{6}}\left\| P_{N_2} F_2\right\|_{L_t^3 L_x^6}\left\|P_{N_3} F_3\right\|_{L_t^6 L_x^{6}}+N_3\left\|  P_{N_1} v_1\right\|_{L_t^2 L_x^{6}}\left\|P_{N_2} F_2\right\|_{L_t^3 L_x^{6}}\left\|  P_{N_3} F_3\right\|_{L_t^6 L_x^{6}}\right]\\
& \lesssim& N\left\|P_{N_1} v_1\right\|_{L_t^2 L_x^4} \left\|P_{N_2} F_2\right\|_{L_t^3 L_x^6}  \cdot N_3^{\frac{1}{3}}\left\|P_{N_3} F_3\right\|_{L_t^6L_x^{6}}\\
&&+ N\cdot N_1\left\|P_{N_1} v_1\right\|_{L_t^2 L_x^4} \left\|P_{N_2} F_2\right\|_{L_t^3 L_x^6}  \cdot N_3^{\frac{1}{3}}\left\|P_{N_3} F_3\right\|_{L_t^6L_x^{6}} \\
&&+ N\cdot N_1^{\frac{1}{3}}\left\|P_{N_1} v_1\right\|_{L_t^2 L_x^4} N_2  \left\|P_{N_2} F_2\right\|_{L_t^3 L_x^6} \left\|P_{N_3} F_3\right\|_{L_t^6L_x^{6}} \\
&&+ N\cdot N_1^{\frac{1}{3}}\left\|P_{N_1} v_1\right\|_{L_t^2 L_x^4}  \left\|P_{N_2} F_2\right\|_{L_t^3 L_x^6} N_3 \left\|P_{N_3} F_3\right\|_{L_t^6L_x^{6}} \\
& \simeq&\left(\frac{N}{N_1}\right)\left(\frac{N_3}{N_2}\right)^{\frac{1}{3}} N_1\left\|P_{N_1} v_1\right\|_{L_t^2 L_x^4} N_2^{\frac{1}{3}} \left\|P_{N_2} F_2\right\|_{L_t^3 L_x^3} \left\|P_{N_3} F_3\right\|_{L_t^6 L_x^{6}}\\
&&+N\left(\frac{N_3}{N_2}\right)^{\frac{1}{3}} N_1\left\|P_{N_1} v_1\right\|_{L_t^2 L_x^4} N_2^{\frac{1}{3}} \left\|P_{N_2} F_2\right\|_{L_t^3 L_x^3} \left\|P_{N_3} F_3\right\|_{L_t^6 L_x^{6}}\\
&&+N\left(\frac{N_2}{N_1}\right)^{\frac{2}{3}} N_1\left\|P_{N_1} v_1\right\|_{L_t^2 L_x^4} N_2^{\frac{1}{3}}\left\|P_{N_2} F_2\right\|_{L_t^3 L_x^6}  \left\|P_{N_3} F_3\right\|_{L_t^6 L_x^{6}}\\
&&+N\left(\frac{N_3}{N_2}\right)  N_1\left\|P_{N_1} v_1\right\|_{L_t^2 L_x^4} N_2^{\frac{1}{3}}\left\|P_{N_2} F_2\right\|_{L_t^3 L_x^6}  \left\|P_{N_3} F_3\right\|_{L_t^6 L_x^{6}}\\
&\lesssim&N\left[\frac{1}{N_1}\left(\frac{N_3}{N_2}\right)^{\frac{1}{3}}+\left(\frac{N_3}{N_2}\right)^{\frac{1}{3}}+\left(\frac{N_2}{N_1}\right)^{\frac{2}{3}}+ \frac{N_3}{N_2}  \right]\left\|P_{N_1} v_1\right\|_{X_{N_1}}\left\|P_{N_2} F_2\right\|_{Y_{N_2}}\left\|P_{N_3} F_3\right\|_{Y_{N_3}}.
\end{eqnarray*}

Next, we will prove the estimates \eqref{eq4.5}-\eqref{eq4.8} where
a(random, low-regularity) forcing term $F$ appears at highest frequency.

\textbf{Case 5:} $\mathbf{F_1F_2F_3}$

Using Lemma \ref{L2.2} to place the $P_{N_3}F_3$ piece into $W_{\mathbf{e}}^{\frac{4}{2-\varepsilon},\frac{4}{\varepsilon}}$, we have
\begin{equation}\label{eq4.9}
  N^{\frac{1}{2}+\varepsilon}\left\|P_N\left(P_{N_1} F_1 P_{N_2} F_2 P_{N_3} F_3\right)\right\|_{W_{\mathbf{e}}^{\frac{4}{4-\varepsilon},\frac{4}{2+\varepsilon}}  } \lesssim N^{\frac{1}{2}+\varepsilon}\left\|P_{N_1} F_1 P_{N_2} F_2\right\|_{L_{\mathbf{e}}^{2,2}}\left\|P_{N_3} F_3\right\|_{W_{\mathbf{e}}^{\frac{4}{2-\varepsilon},\frac{4}{\varepsilon}}}.
\end{equation}
According to the identity \eqref{eq2.1}, we decompose the highest frequency piece $P_{N_1} F_1$ into
\begin{eqnarray*}
P_{N_1} F_1&= & P_{N_1, \mathbf{e}_1} P_{N_1} F_1+P_{N_1, \mathbf{e}_2}\left(1-P_{N_1, \mathbf{e}_1}\right) P_{N_1} F_1 +P_{N_1, \mathbf{e}_3}\left(1-P_{N_1, \mathbf{e}_2}\right)\left(1-P_{N_1, \mathbf{e}_1}\right) P_{N_1} F_1 \\
&& +P_{N_1, \mathbf{e}_4}\left(1-P_{N_1, \mathbf{e}_3}\right)\left(1-P_{N_1, \mathbf{e}_2}\right)\left(1-P_{N_1, \mathbf{e}_1}\right) P_{N_1} F_1 .
\end{eqnarray*}
We note that the operators $\left(1-P_{N_1, \mathbf{e}_{\ell}}\right)P_{N_1}$ are disposable since their kernels are uniformly bounded in $L_x^1$ and the fact that $L_{\mathbf{e}}^{2,2}=L_{\mathbf{e}_{\ell}}^{2,2}$ for $\ell=1, \ldots, 4$ by Fubini's theorem, we may now use H\"older's inequality to bound \eqref{eq4.9} by
\begin{eqnarray*}
&& N^{\frac{1}{2}+\varepsilon}\left\|P_{N_1} F_1 P_{N_2} F_2\right\|_{L_{\mathbf{e}}^{2,2}}\left\|P_{N_3} F_3\right\|_{W_{\mathbf{e}}^{\frac{4}{2-\varepsilon},\frac{4}{\varepsilon}}}\\
&\lesssim&N^{\frac{1}{2}+\varepsilon}\sum_{\ell=1}^4\left\|P_{N_1, \mathrm{e}_{\ell}} P_{N_1} F_1\right\|_{L_{\mathbf{e}_{\ell}}^{\frac{4}{\varepsilon}, \frac{4}{2-\varepsilon}}}\left\|P_{N_2} F_2\right\|_{L_{\mathbf{e}_{\ell}}^{\frac{4}{2-\varepsilon},\frac{4}{\varepsilon}}}\left\|P_{N_3}F_3\right\|_{W_{\mathbf{e}}^{\frac{4}{2-\varepsilon},\frac{4}{\varepsilon}}}\\
&\simeq&\left(\frac{N}{N_1}\right)^{\frac{1}{2}+\varepsilon}\left(\frac{N_2}{N_1}\right)^{\frac{1}{6}}\left(\frac{N_3}{N_1}\right)^{\frac{1}{6}} \sum_{\ell=1}^4 N_1^{\frac{5}{6}+\varepsilon}\left\|P_{N_1, \mathrm{e}_{\ell}} P_{N_1} F_1\right\|_{L_{\mathbf{e}_{\ell}}^{\frac{4}{\varepsilon}, \frac{4}{2-\varepsilon}}}\cdot N_2^{-\frac{1}{6}}\left\|P_{N_2} F_2\right\|_{L_{\mathbf{e}_{\ell}}^{\frac{4}{2-\varepsilon},\frac{4}{\varepsilon}}} \\
&&\cdot N_3^{-\frac{1}{6}}\left\|P_{N_3}F_3\right\|_{W_{\mathbf{e}}^{\frac{4}{2-\varepsilon},\frac{4}{\varepsilon}}}\\
&\lesssim&\left(\frac{N}{N_1}\right)^{\frac{1}{2}+\varepsilon}\left(\frac{N_2}{N_1}\right)^{\frac{1}{6}}\left(\frac{N_3}{N_1}\right)^{\frac{1}{6}} \sum_{\ell=1}^4 N_1^{\frac{5}{6}+\varepsilon}\left\|P_{N_1, \mathrm{e}_{\ell}} P_{N_1} F_1\right\|_{L_{\mathbf{e}_{\ell}}^{\frac{4}{\varepsilon}, \frac{4}{2-\varepsilon}}}\cdot N_2^{-\frac{1}{6}}\left\|P_{N_2} F_2\right\|_{W_{\mathbf{e}_{\ell}}^{\frac{4}{2-\varepsilon},\frac{4}{\varepsilon}}} \\
&&\cdot N_3^{-\frac{1}{6}}\left\|P_{N_3}F_3\right\|_{W_{\mathbf{e}}^{\frac{4}{2-\varepsilon},\frac{4}{\varepsilon}}}\\
&\lesssim&\left(\frac{N}{N_1}\right)^{\frac{1}{2}+\varepsilon}\left(\frac{N_3}{N_1}\right)^{\frac{1}{6}}\left\|P_{N_1} F_1\right\|_{Y_{N_1}}\left\|P_{N_2} F_2\right\|_{Y_{N_2}}\left\|P_{N_3} F_3\right\|_{Y_{N_3}} .
\end{eqnarray*}

\textbf{Case 6:} $\mathbf{F_1v_2v_3}$

We begin by placing the $P_{N_3} v_3$ piece into $W_{\mathbf{e}}^{\frac{4}{2-\varepsilon},\frac{4}{\varepsilon}}$,
\begin{equation}\label{eq4.10}
  N^{\frac{1}{2}+\varepsilon}\left\|P_N\left(P_{N_1} F_1 P_{N_2} v_2 P_{N_3} v_3\right)\right\|_{W_{\mathbf{e}}^{\frac{4}{4-\varepsilon},\frac{4}{2+\varepsilon}}  } \lesssim N^{\frac{1}{2}+\varepsilon}\left\|P_{N_1} F_1 P_{N_2} v_2\right\|_{L_{\mathbf{e}}^{2,2}}\left\|P_{N_3} v_3\right\|_{W_{\mathbf{e}}^{\frac{4}{2-\varepsilon},\frac{4}{\varepsilon}}}.
\end{equation}
On the one hand, similar to \eqref{eq4.5},
\begin{equation*}
  \left\|P_{N_1} F_1 P_{N_2} v_2\right\|_{L_{\mathbf{e}}^{2,2}} \lesssim \sum_{\ell=1}^4\left\|P_{N_1, \mathbf{e}_{\ell}} P_{N_1} F_1\right\|_{L_{\mathbf{e}_{\ell}}^{\frac{4}{\varepsilon}, \frac{4}{2-\varepsilon}}}\cdot  \sum_{\ell=1}^4\left\|P_{N_2} v_2\right\|_{L_{\mathbf{e}_{\ell}}^{\frac{4}{2-\varepsilon},\frac{4}{\varepsilon}}} ,
\end{equation*}
On the other hand, by using $L_{\mathbf{e}}^{2,2}=L_t^2L_x^2$, we get
\begin{equation*}
  \left\|P_{N_1} F_1 P_{N_2} v_2\right\|_{L_{\mathbf{e}}^{2,2}} \lesssim  \left\|P_{N_1} F_1\right\|_{L_t^6L_x^6}\cdot  \left\|P_{N_2} v_2\right\|_{L_t^3L_x^3}.
\end{equation*}
Now, by interpolation method, we can estimate \eqref{eq4.10} by
\begin{eqnarray*}
&&N^{\frac{1}{2}+\varepsilon}\left\|P_{N_1} F_1 P_{N_2} v_2\right\|_{L_{\mathbf{e}}^{2,2}}\left\|P_{N_3} v_3\right\|_{W_{\mathbf{e}}^{\frac{4}{2-\varepsilon}, \frac{4}{\varepsilon}}}\\
&\lesssim&N^{\frac{1}{2}+\varepsilon}\left(\sum_{\ell=1}^4\left\|P_{N_1, \mathbf{e}_{\ell}} P_{N_1} F_1\right\|_{L_{\mathbf{e}_{\ell}}^{\frac{4}{\varepsilon}, \frac{4}{2-\varepsilon}}}\right)^{\frac{2(1-3\varepsilon)}{3(1-2\varepsilon)}}\cdot\left( \left\|P_{N_1} F_1\right\|_{L_t^6 L_x^6}\right)^{1-\frac{2(1-3\varepsilon)}{3(1-2\varepsilon)}}\cdot\left( \left\|P_{N_2} v_2\right\|_{L_t^3 L_x^3}\right)^{1-\frac{2(1-3\varepsilon)}{3(1-2\varepsilon)}}\\
&&\cdot\left(\sum_{\ell=1}^4\left\|P_{N_2} v_2\right\|_{L_{\mathbf{e}_{\ell}}^{\frac{4}{2-\varepsilon},\frac{4}{\varepsilon}}}\right)^{\frac{2(1-3\varepsilon)}{3(1-2\varepsilon)}}\cdot\left\|P_{N_3} v_3\right\|_{W_{\mathbf{e}}^{\frac{4}{2-\varepsilon},\frac{4}{\varepsilon}}}\\
&\lesssim&\left(\frac{N}{N_2}\right)^{\frac{1}{2}+\varepsilon}\left(\frac{N_3}{N_2}\right)^{\frac{1}{2}-\varepsilon}\cdot \left(\sum_{\ell=1}^4 N_1^{\frac{5}{6}+  \varepsilon}\left\|P_{N_1, \mathbf{e}_{\ell}} P_{N_1} F_1\right\|_{L_{\mathbf{e}_{\ell}}^{\frac{4}{\varepsilon}, \frac{4}{2-\varepsilon}}}\right)^{\frac{2(1-3\varepsilon)}{3(1-2\varepsilon)}}\cdot N_3^{-\frac{1}{2}+\varepsilon}\left\|P_{N_3} v_3\right\|_{W_{\mathbf{e}}^{\frac{4}{2-\varepsilon},\frac{4}{\varepsilon}}}   \\
&& \cdot \left( \left\|P_{N_1} F_1\right\|_{L_t^6 L_x^6}\right)^{1-\frac{2(1-3\varepsilon)}{3(1-2\varepsilon)}}\cdot\left(\sum_{\ell=1}^4N_2^{-\frac{1}{2}+\varepsilon} \left\|P_{N_2} v_2\right\|_{L_{\mathbf{e}_{\ell}}^{ \frac{4}{2-\varepsilon},\frac{4}{\varepsilon}}}\right)^{\frac{2(1-3\varepsilon)}{3(1-2\varepsilon)}}\cdot\left(N_2 \left\|P_{N_2} v_2\right\|_{L_t^3 L_x^3}\right)^{1-\frac{2(1-3\varepsilon)}{3(1-2\varepsilon)}}
\\
&\lesssim&\left(\frac{N}{N_2}\right)^{\frac{1}{2}+\varepsilon}\left(\frac{N_3}{N_2}\right)^{\frac{1}{2}-\varepsilon}\left\|P_{N_1} F_1\right\|_{Y_{N_1}}\left\|P_{N_2} v_2\right\|_{X_{N_2}}\left\|P_{N_3} v_3\right\|_{X_{N_3}} .
\end{eqnarray*}

\textbf{Case 7:} $\mathbf{F_1F_2v_3}$

We begin by placing the $P_{N_2} F_2$ piece into $W_{\mathbf{e}}^{\frac{4}{2-\varepsilon},\frac{4}{\varepsilon}}$,
\begin{equation}\label{eq4.11}
  N^{\frac{1}{2}+\varepsilon}\left\|P_N\left(P_{N_1} F_1 P_{N_2} F_2 P_{N_3} v_3\right)\right\|_{W_{\mathbf{e}}^{\frac{4}{4-\varepsilon},\frac{4}{2+\varepsilon}}  } \lesssim N^{\frac{1}{2}+\varepsilon}\left\|P_{N_1} F_1 P_{N_3} v_3\right\|_{L_{\mathbf{e}}^{2,2}}\left\|P_{N_2} F_2\right\|_{W_{\mathbf{e}}^{\frac{4}{2-\varepsilon},\frac{4}{\varepsilon}}}.
\end{equation}
Interpolating as in the proof of \eqref{eq4.6}, we can estimate \eqref{eq4.11} by
\begin{eqnarray*}
&& N^{\frac{1}{2}+\varepsilon}\left\|P_{N_1} F_1 P_{N_3} v_3\right\|_{L_{\mathbf{e}}^{2,2}}\left\|P_{N_2} F_2\right\|_{W_{\mathbf{e}}^{\frac{4}{2-\varepsilon},\frac{4}{\varepsilon}}}\\
& \lesssim& N^{\frac{1}{2}+\varepsilon}\left(\sum_{\ell=1}^4\left\|P_{N_1, \mathbf{e}_{\ell}} P_{N_1} F_1\right\|_{L_{\mathbf{e}_{\ell}}^{\frac{4}{\varepsilon}, \frac{4}{2-\varepsilon}}}\right)^{\frac{1-3\varepsilon}{2-3\varepsilon}}\left(\left\|P_{N_1} F_1\right\|_{L_t^6 L_x^6}\right)^{1-\frac{1-3\varepsilon}{2-3\varepsilon}} \cdot\left\|P_{N_2} F_2\right\|_{W_{\mathbf{e}}^{\frac{4}{2-\varepsilon},\frac{4}{\varepsilon}}}\\
&&\cdot\left(\sum_{\ell=1}^4\left\| P_{N_3}v_3\right\|_{L_{\mathbf{e}_{\ell}}^{ \frac{4}{2-\varepsilon},\frac{4}{\varepsilon}}}\right)^{\frac{1-3\varepsilon}{2-3\varepsilon}}\left(\left\|P_{N_3} v_3\right\|_{L_t^3 L_x^3}\right)^{1-\frac{1-3\varepsilon}{2-3\varepsilon}}\\
&\lesssim&\left(\frac{N}{N_3}\right)^{\frac{1}{2}+\varepsilon}\left(\frac{N_2}{N_1}\right)^{\frac{1}{6}} \cdot\left(\sum_{\ell=1}^4 N_1^{\frac{5}{6}+\varepsilon}\left\|P_{N_1, \mathbf{e}_{\ell}} P_{N_1} F_1\right\|_{L_{\mathbf{e}_{\ell}}^{\frac{4}{\varepsilon}, \frac{4}{2-\varepsilon}}}\right)^{\frac{1-3\varepsilon}{2-3\varepsilon}}\cdot N_2^{-\frac{1}{6}}\left\|P_{N_2} F_2\right\|_{W_{\mathbf{e}}^{\frac{4}{2-\varepsilon},\frac{4}{\varepsilon}}}\\
&& \left(\left\|P_{N_1} F_1\right\|_{L_t^6 L_x^6}\right)^{1-\frac{1-3\varepsilon}{2-3\varepsilon}}\left(\sum_{\ell=1}^4 N_3^{-\frac{1}{2}+\varepsilon}\left\|P_{N_3} v_3\right\|_{L_{\mathbf{e}_{\ell}}^{ \frac{4}{2-\varepsilon},\frac{4}{\varepsilon}}}\right)^{\frac{1-3\varepsilon}{2-3\varepsilon}}\left(N_3\left\|P_{N_3} v_3\right\|_{L_t^3 L_x^3}\right)^{1-\frac{1-3\varepsilon}{2-3\varepsilon}} \\
&\lesssim&\left(\frac{N}{N_3}\right)^{\frac{1}{2}+\varepsilon}\left(\frac{N_2}{N_1}\right)^{\frac{1}{6}} \cdot\left\|P_{N_1} F_1\right\|_{Y_{N_1}}\left\|P_{N_2} F_2\right\|_{Y_{N_2}}\left\|P_{N_3} v_3\right\|_{X_{N_3}} .
\end{eqnarray*}

\textbf{Case 8:} $\mathbf{F_1v_2F_3}$

We begin by placing the $P_{N_3} F_3$ piece into $W_{\mathbf{e}}^{\frac{4}{2-\varepsilon},\frac{4}{\varepsilon}}$,
\begin{equation}\label{eq4.12}
  N^{\frac{1}{2}+\varepsilon}\left\|P_N\left(P_{N_1} F_1 P_{N_2} v_2 P_{N_3} F_3\right)\right\|_{W_{\mathbf{e}}^{\frac{4}{4-\varepsilon},\frac{4}{2+\varepsilon}}  } \lesssim N^{\frac{1}{2}+\varepsilon}\left\|P_{N_1} F_1 P_{N_2} v_2\right\|_{L_{\mathbf{e}}^{2,2}}\left\|P_{N_3} F_3\right\|_{W_{\mathbf{e}}^{\frac{4}{2-\varepsilon},\frac{4}{\varepsilon}}}.
\end{equation}
Analogously to the previous two cases,  we can estimate \eqref{eq4.12} by
\begin{eqnarray*}
&& N^{\frac{1}{2}+\varepsilon}\left\|P_{N_1} F_1 P_{N_2} v_2\right\|_{L_{\mathbf{e}}^{2,2}}\left\|P_{N_3} F_3\right\|_{W_{\mathbf{e}}^{\frac{4}{2-\varepsilon},\frac{4}{\varepsilon}}} \\
& \lesssim& N^{\frac{1}{2}+\varepsilon}\left(\sum_{\ell=1}^4\left\|P_{N_1, \mathbf{e}_{\ell}} P_{N_1} F_1\right\|_{L_{\mathbf{e}_{\ell}}^{\frac{4}{\varepsilon}, \frac{4}{2-\varepsilon}}}\right)^{\frac{1+3\varepsilon}{1+6\varepsilon}}\cdot\left(\left\|P_{N_1} F_1\right\|_{L_t^6 L_x^6}\right)^{1-\frac{1+3\varepsilon}{1+6\varepsilon}} \cdot\left(\sum_{\ell=1}^4\left\|P_{N_2} v_2\right\|_{L_{\mathbf{e}_{\ell}}^{\frac{4}{2-\varepsilon}, \frac{4}{\varepsilon}}}\right)^{\frac{1+3\varepsilon}{1+6\varepsilon}}\\
&&\cdot\left(\left\|P_{N_2} v_2\right\|_{L_t^3 L_x^3}\right)^{1-\frac{1+3\varepsilon}{1+6\varepsilon}}\cdot\left\|P_{N_3} F_3\right\|_{W_{\mathbf{e}}^{\frac{4}{2-\varepsilon},\frac{4}{\varepsilon}}}\\
&\lesssim&\left(\frac{N}{N_2}\right)^{\frac{1}{2}+\varepsilon}\left(\frac{N_3}{N_2}\right)^{\frac{1}{6} }\cdot\left(\sum_{\ell=1}^4 N_1^{\frac{5}{6}+\varepsilon}\left\|P_{N_1, \mathbf{e}_{\ell}} P_{N_1} F_1\right\|_{L_{\mathbf{e}_{\ell}}^{\frac{4}{\varepsilon}, \frac{4}{2-\varepsilon}}}\right)^{\frac{1+3\varepsilon}{1+6\varepsilon}}\left(\left\|P_{N_1} F_1\right\|_{L_t^6 L_x^6}\right)^{1-\frac{1+3\varepsilon}{1+6\varepsilon}} \\
&&\left(\sum_{\ell=1}^4N_2^{-\frac{1}{2}+\varepsilon}\left\|P_{N_2} v_2\right\|_{L_{\mathbf{e}_{\ell}}^{\frac{4}{2-\varepsilon}, \frac{4}{\varepsilon}}}\right)^{\frac{1+3\varepsilon}{1+6\varepsilon}}\cdot\left(N_2\left\|P_{N_2} v_2\right\|_{L_t^3 L_x^3}\right)^{1-\frac{1+3\varepsilon}{1+6\varepsilon}}\cdot N_3^{-\frac{1}{6}}\left\|P_{N_3} F_3\right\|_{W_{\mathbf{e}}^{\frac{4}{2-\varepsilon},\frac{4}{\varepsilon}}}\\
&\lesssim&\left(\frac{N}{N_2}\right)^{\frac{1}{2}+\varepsilon}\left(\frac{N_3}{N_2}\right)^{\frac{1}{6} }\left\|P_{N_1} F_1\right\|_{Y_{N_1}}\left\|P_{N_2} v_2\right\|_{X_{N_2}}\left\|P_{N_3} F_3\right\|_{Y_{N_3}} .
\end{eqnarray*}

\end{proof}
The frequency localized, trilinear estimates \eqref{eq4.1}-\eqref{eq4.8} imply an important set of nonlinear estimates that we will need for the proofs of the almost sure local well-posedness result of Theorem \ref{t1.1}. More precisely, given any time interval $I$, any forcing term $F \in Y(I)$, and any $v, v_1, v_2, u, w \in X(I)$, it is an easy consequence of the exponential gains in the frequency differences in the trilinear estimates \eqref{eq4.1}-\eqref{eq4.8} to conclude that
\begin{eqnarray}\label{eq4.13}
\left\||F+v|^2(F+v)\right\|_{G(I)} &\lesssim&\|v\|_{X(I)}^3+\|v\|_{X(I)}^2\|F\|_{Y(I)}+\|v\|_{X(I)}\|F\|_{Y(I)}^2+\|F\|_{Y(I)}^3 \nonumber\\
& \lesssim&\|F\|_{Y(I)}^3+\|v\|_{X(I)}^3
\end{eqnarray}
as well as
\begin{eqnarray}\label{eq4.14}
&&\| |F+ v_1|^2(F+v_1)-|F+v_2|^2(F+v_2) \|_{G(I)} \nonumber\\
 &\lesssim&\left\|v_1-v_2\right\|_{X(I)}\left(\|F\|_{Y(I)}^2+\left\|v_1\right\|_{X(I)}^2+\left\|v_2\right\|_{X(I)}^3\right).
\end{eqnarray}
Moreover, using that
\begin{eqnarray*}
|F+u+w|^2(F+u+w)-|u|^2 u&=& \left.F\right|^2 F+|w|^2 w+O\left(F^2 u\right)+O\left(F^2 w\right)+O\left(F u^2\right) \\
&& +O(F u w)+O\left(F w^2\right)+O\left(u^2 w\right)+O\left(u w^2\right),
\end{eqnarray*}
we may also infer that
\begin{eqnarray}\label{eq4.15}
&&\left\||F+u+w|^2(F+u+w)-|u|^2 u\right\|_{G(I)} \nonumber\\
& \lesssim&\|F\|_{Y(I)}^3+\|w\|_{X(I)}^3+\|F\|_{Y(I)}\|u\|_{X(I)}^2+\|u\|_{X(I)}^2\|w\|_{X(I)}.
\end{eqnarray}

\section{Almost sure bounds}
In this section we establish various almost sure bounds for the free evolution of the
random data.
\begin{lemma}{\em(see \cite[Lemma 3.1]{{NBI2008}})}\label{L5.1}
Let $\left\{g_n\right\}_{n=1}^{\infty}$ be a sequence of real-valued, independent, zero-mean random variables with associated distributions $\left\{\mu_n\right\}_{n=1}^{\infty}$ on a probability space $(\Omega, \mathcal{A}, \mathbb{P})$. Assume that the distributions satisfy the property that there exists $c>0$ such that
$$
\left|\int_{-\infty}^{+\infty} e^{\gamma x} d \mu_n(x)\right| \leq e^{c \gamma^2} \text { for all } \gamma \in \mathbb{R} \text { and for all } n \in \mathbb{N}.
$$
Then there exists $\alpha>0$ such that for every $\lambda>0$ and every sequence $\left\{c_n\right\}_{n=1}^{\infty} \in \ell^2(\mathbb{N} ; \mathbb{C})$ of complex numbers,
$$
\mathbb{P}\left(\left\{\omega:\left|\sum_{n=1}^{\infty} c_n g_n(\omega)\right|>\lambda\right\}\right) \leq 2 \exp \left(-\alpha \frac{\lambda^2}{\sum_n\left|c_n\right|^2}\right).
$$
As a consequence there exists $C>0$ such that for every $2 \leq p<\infty$ and every $\left\{c_n\right\}_{n=1}^{\infty} \in$ $\ell^2(\mathbb{N} ; \mathbb{C})$
$$
\left\|\sum_{n=1}^{\infty} c_n g_n(\omega)\right\|_{L_\omega^p(\Omega)} \leq C \sqrt{p}\left(\sum_{n=1}^{\infty}\left|c_n\right|^2\right)^{1 / 2}.
$$
\end{lemma}

\begin{lemma}{\em(see \cite[Lemma 5.2]{{BDJL2019}})}\label{L5.2}
Let $F$ be a real-valued measurable function on a probability space $(\Omega, \mathcal{A}, \mathbb{P})$. Suppose that there exist $C_0>0, K>0$ and $p_0 \geq 1$ such that for every $p \geq p_0$ we have
$$
\|F\|_{L_\omega^p(\Omega)} \leq \sqrt{p} C_0 K.
$$
Then there exist $c>0$ and $C_1>0$, depending on $C_0$ and $p_0$ but independent of $K$, such that for every $\lambda>0$,
\begin{equation*}
  \mathbb{P}(\{\omega \in \Omega:|F(\omega)|>\lambda\}) \leq C_1 e^{-c\frac{ \lambda^2}{ K^2}}.
\end{equation*}
In particular, it follows that
\begin{equation*}
  \mathbb{P}(\{\omega \in \Omega:|F(\omega)|<\infty\})=1.
\end{equation*}
\end{lemma}
Using the same technique as in \cite{BDJL2019}, we can easily obtain the following lemma.
\begin{lemma}\label{L5.3}
There exists an constant $C \geq 1$ such that for all $k \in \mathbb{Z}^4$ with $|k| \geq 10$ and for each $\ell=1, \ldots, 4$, it holds that
\begin{equation}\label{eq5.1}
  \left\|e^{i t H} P_k f\right\|_{L_{\mathbf{e}_{\ell}}^{2, \infty}\left(\mathbb{R} \times \mathbb{R}^4\right)} \leq C|k|^{\frac{1}{2}}\left\|P_k f\right\|_{L_x^2\left(\mathbb{R}^4\right)}.
\end{equation}
\end{lemma}

\begin{lemma}\label{L5.4}
Let $\frac{1}{3}<s<1$ and let $0<\varepsilon<\frac{1}{3}\left(s-\frac{1}{3}\right)$. Let $f \in H_x^s(\mathbb{R}^4)$ and denote by $f^\omega$ the randomization of $f$ as defined in \eqref{eq2.99}. Then there exist absolute constants $C>0$ and $c>0$ such that for any $\lambda>0$ it holds that
$$
\mathbb{P}\left(\left\{\omega \in \Omega:\left\|e^{i t H} f^\omega\right\|_{Y(\mathbb{R})}>\lambda\right\}\right) \leq C \exp \left(-c \lambda^2\|f\|_{H_x^s\left(\mathbb{R}^4\right)}^{-2}\right).
$$
In particular, we have for almost every $\omega \in \Omega$ that
$$
\left\|e^{i t H} f^\omega\right\|_{Y(\mathbb{R})}<\infty.
$$
\end{lemma}
\begin{proof}
For any $p \geq \frac{4}{\varepsilon}$ we have by Minkowski's inequality that
$$
\left\|\left\|e^{i t H} f^\omega\right\|_{Y(\mathbb{R})}\right\|_{L_\omega^p} \leq\left(\sum_{N \in 2^\mathbb{Z}}\| \| P_N e^{i t H} f^\omega\left\|_{Y_N(\mathbb{R})}\right\|_{L_\omega^p}^2\right)^{\frac{1}{2}}.
$$
We consider the components of the $Y_N(\mathbb{R})$ norm separately. In the sequel, we will repeatedly use that the free evolution and the frequency projections commute. We first treat the component coming from the Strichartz spaces. By Proposition \ref{P3.1}, the unit-scale Bernstein estimate and Lemma \ref{L3.1}, we have
\begin{eqnarray*}
\langle N\rangle^{\frac{1}{3}+3 \varepsilon}\left\| \left\| P_N e^{i t H}f^\omega\right\|_{L_t^3 L_x^6 \cap L_t^6 L_x^6\left(\mathbb{R} \times \mathbb{R}^4\right)}\right\|_{L_\omega^p}
& \leq&\langle N\rangle^{\frac{1}{3}+3 \varepsilon}\left\| \| \sum_{|k| \sim N} g_k(\omega) e^{i t H} P_k f\|_{L_\omega^p}\right\|_{L_t^3 L_x^6 \cap L_t^6 L_x^6\left(\mathbb{R} \times \mathbb{R}^4\right)} \\
& \lesssim& \sqrt{p}\langle N\rangle^{\frac{1}{3}+3 \varepsilon}\left(\sum_{|k| \sim N}\left\|e^{i t H} P_k f\right\|_{L_t^3 L_x^6 \cap L_t^6 L_x^6\left(\mathbb{R} \times \mathbb{R}^4\right)}^2\right)^{\frac{1}{2}} \\
& \lesssim& \sqrt{p}\langle N\rangle^{\frac{1}{3}+3 \varepsilon}\left(\sum_{|k| \sim N}\left\|e^{i t H} P_k f\right\|_{L_t^3 L_x^3 \cap L_t^6 L_x^{\frac{12}{5}}\left(\mathbb{R} \times \mathbb{R}^4\right)}^2\right)^{\frac{1}{2}} \\
& \lesssim& \sqrt{p}\langle N\rangle^{\frac{1}{3}+3 \varepsilon}\left(\sum_{|k| \sim N}\left\|P_k f\right\|_{L_x^2\left(\mathbb{R}^4\right)}^2\right)^{\frac{1}{2}} \\
& \simeq& \sqrt{p}\langle N\rangle^{\frac{1}{3}+3 \varepsilon}\left\|P_N f\right\|_{L_x^2\left(\mathbb{R}^4\right)} .
\end{eqnarray*}
Next, we estimate the local smoothing type component of the $Y_N(\mathbb{R})$ component. Here we first apply the local smoothing type estimate \eqref{eq3.5} for the lateral spaces and then use the large deviation estimate from Lemma \ref{L5.3} to obtain that
\begin{eqnarray*}
&& \left\|\sum_{\ell=1}^4\langle N\rangle^{\frac{1}{3}+3 \varepsilon} N^{\frac{1}{2}-\varepsilon}\| e^{i t H} P_{N, \mathbf{e}_{\ell}} P_N f^\omega\|_{L_{\mathbf{e}_{\ell}}^{\frac{4}{\varepsilon}, \frac{4}{2-\varepsilon}}\left(\mathbb{R} \times \mathbb{R}^4\right)}\right\|_{L_\omega^p} \\
&  \lesssim&\left\|\langle N\rangle^{\frac{1}{3}+3 \varepsilon}\| P_N f^\omega \|_{L_x^2\left(\mathbb{R}^4\right)}\right\|_{L_\omega^p} \\
& \lesssim& \sqrt{p}\langle N\rangle^{\frac{1}{3}+3 \varepsilon}\left(\sum_{|k| \sim N}\left\|P_k f\right\|_{L_x^2\left(\mathbb{R}^4\right)}^2\right)^{\frac{1}{2}} \\
& \lesssim& \sqrt{p}\langle N\rangle^{\frac{1}{3}+3 \varepsilon}\left\|P_N f\right\|_{L_x^2\left(\mathbb{R}^4\right)}.
\end{eqnarray*}
Finally, we turn to the maximal function type component of the $Y_N(\mathbb{R})$ norm, where we distinguish the large frequency regime $N \gtrsim 1$ and the small frequency regime $N \lesssim 1$. For large frequencies $N \gtrsim 1$ we first use the large deviation estimate from Lemma \ref{L5.1} and then interpolate between the improved maximal function estimate \eqref{eq5.1} for unit-scale frequency localized data and an estimate of the $L_t^4 L_x^4\left(\mathbb{R} \times \mathbb{R}^4\right)$ norm of the free evolution of unit-scale frequency localized data (based on the unit-scale Bernstein estimate and Strichartz estimates) to conclude that
\begin{eqnarray*}
\left\|\sum_{\ell=1}^4\langle N\rangle^{-\frac{1}{6}} \| e^{i t H}  P_N f^\omega\|_{L_{\mathbf{e}_{\ell}}^{\frac{4}{2-\varepsilon}, \frac{4}{\varepsilon}}\left(\mathbb{R} \times \mathbb{R}^4\right)}\right\|_{L_\omega^p}& \lesssim& \sqrt{p}\sum_{\ell=1}^4\langle N\rangle^{-\frac{1}{6}}\left(\sum_{|k| \sim N}\left\|e^{i t H} P_k f\right\|_{L_{\mathbf{e}_{\ell}}^{\frac{4}{2-\varepsilon}, \frac{4}{\varepsilon}}\left(\mathbb{R} \times \mathbb{R}^4\right)}^2\right)^{\frac{1}{2}}\\
& \lesssim& \sqrt{p} N^{-\frac{1}{6}} N^{\frac{1}{2}}\left(\sum_{|k| \sim N}\left\|P_k f\right\|_{L_x^2\left(\mathbb{R}^4\right)}^2\right)^{\frac{1}{2}} \\
& \lesssim& \sqrt{p}\langle N\rangle^{\frac{1}{3}}\left\|P_N f\right\|_{L_x^2\left(\mathbb{R}^4\right)} .
\end{eqnarray*}
For small frequencies $N \lesssim 1$, we directly apply \eqref{eq3.4}, trivially bounding the resulting frequency factors, and then use the large deviation estimate to infer that in this case
\begin{eqnarray*}
 \left\|\sum_{\ell=1}^4\langle N\rangle^{-\frac{1}{6}} \| e^{i t H}  P_N f^\omega\|_{L_{\mathbf{e}_{\ell}}^{\frac{4}{2-\varepsilon}, \frac{4}{\varepsilon}}\left(\mathbb{R} \times \mathbb{R}^4\right)}\right\|_{L_\omega^p} &\lesssim& N^{-\frac{1}{6}} N^{\frac{3}{2}-\varepsilon}\left\| \|  P_N f^\omega \|_{L_x^2\left(\mathbb{R}^4\right)}\right\|_{L_\omega^p} \\
& \lesssim& \sqrt{p}\left\|P_N f\right\|_{L_x^2\left(\mathbb{R}^4\right)} .
\end{eqnarray*}
Putting all of the above estimates together, we find that
$$
\left\|\left\|e^{i t H} f^\omega\right\|_{Y(\mathbb{R})}\right\|_{L_\omega^p} \lesssim \sqrt{p}\left(\sum_{N \in 2^\mathbb{Z}}\left(\langle N\rangle^{\frac{1}{3}+3 \varepsilon}\left\|P_N f\right\|_{L_x^2\left(\mathbb{R}^4\right)}\right)^2\right)^{\frac{1}{2}} \lesssim \sqrt{p}\|f\|_{H_x^s\left(\mathbb{R}^4\right)},
$$
from which the assertion follows by Lemma \ref{L5.2}.
\end{proof}
\section{The proof of Theorem \ref{t1.1}}
In this section we establish local well-posedness for the forced cubic NLS
\begin{equation}\label{eq6.1}
  \left\{\begin{aligned}
\left(i \partial_t+\Delta-V\right) v & = \pm|F+v|^2(F+v) \\
v\left(t_0\right) & =v_0 \in H_x^1(\mathbb{R}^4)
\end{aligned}\right.
\end{equation}
for forcing terms $F: \mathbb{R} \times \mathbb{R}^4 \rightarrow \mathbb{C}$ satisfying $\|F\|_{Y(\mathbb{R})}<\infty$. Recall that by Lemma \ref{L5.4} we have $\left\|e^{i t H} f^\omega\right\|_{Y(\mathbb{R})}<\infty$ almost surely for any $f \in H_x^s\left(\mathbb{R}^4\right)$ with $\frac{1}{3}<s<1$. The proof of the almost sure local well-posedness result of Theorem \ref{t1.1} is then an immediate consequence of the following local well-posedness result for \eqref{eq6.1}.

\begin{proposition}\label{P6.1} Let $t_0 \in \mathbb{R}$ and let $I$ be an open time interval containing $t_0$. Let $F \in$ $Y(\mathbb{R})$ and let $v_0 \in H_x^1(\mathbb{R}^4)$. There exists $0<\delta \ll 1$ such that if
\begin{equation}\label{eq6.2}
  \left\|e^{i\left(t-t_0\right) H} v_0\right\|_{X(I)}+\|F\|_{Y(I)} \leq \delta,
\end{equation}
then there exists a unique solution
$$
v \in C(I ; H_x^1(\mathbb{R}^4)) \cap X(I)
$$
to \eqref{eq6.1} on $I \times \mathbb{R}^4$. Moreover, the solution extends to a unique solution $v: I_* \times \mathbb{R}^4 \rightarrow \mathbb{C}$ to the Cauchy problem \eqref{eq6.1} with maximal time interval of existence $I_*(t_0\in I_*)$, and we have the finite time blowup criterion
$$
\sup I_*<\infty \quad \Rightarrow \quad\|v\|_{X([t_0, \sup I_*))}=+\infty
$$
with an analogous statement in the negative time direction. Finally, a global solution $v(t)$ to \eqref{eq6.1} satisfying $\|v\|_{X(\mathbb{R})}<\infty$ scatters as $t \rightarrow \pm \infty$ in the sense that there exist states $v^{ \pm} \in H_x^1(\mathbb{R}^4)$ such that
$$
\lim _{t \rightarrow \pm \infty}\left\|v(t)-e^{i t \Delta} v^{ \pm}\right\|_{H_x^1\left(\mathbb{R}^4\right)}=0.
$$
\end{proposition}
\begin{proof}
\textbf{Existence of solution.} Without loss of generality we may assume that $t_0=0$. Let $I(0\in I)$ be an open time interval for which \eqref{eq6.2} holds. Note that the existence of such an interval follows from Proposition \ref{P3.2}(i) and the assumption that $\|F\|_{Y(\mathbb{R})}<\infty$. We construct the desired local solution via a standard contraction mapping argument. Let $\delta>0$ be an absolute constant whose size will be chosen sufficiently small further below. We define the ball
$$
\mathcal{B}:=\left\{v \in X(I):\|(1+H)^{\frac{1}{2}}v\|_{X(I)} \leq 2 \delta\right\}
$$
and the map
$$
\Phi(v)(t):=e^{i t H} v_0 - i \int_0^t e^{i(t-s) H}|F+v|^2(F+v)(s) d s.
$$
From our main linear estimate \eqref{eq3.16} and the nonlinear estimates \eqref{eq4.1}-\eqref{eq4.2}, upon choosing $\delta:=(18 C)^{-\frac{1}{2}}$, we obtain for any $v, v_1, v_2 \in \mathcal{B}$ that
$$
\begin{aligned}
\|(1+H)^{\frac{1}{2}}\Phi(v)\|_{X(I)} & \leq\left\|(1+H)^{\frac{1}{2}}e^{i t H} v_0\right\|_{X(I)}+\left\|\int_0^t (1+H)^{\frac{1}{2}}e^{i(t-s) H}|F+v|^2(F+v)(s) d s\right\|_{X(I)} \\
& \leq\left\|e^{i t \Delta} v_0\right\|_{X(I)}+C\left\|(1+H)^{\frac{1}{2}}|F+v|^2(F+v)\right\|_{G(I)} \\
& \leq\left\|e^{i t \Delta} v_0\right\|_{X(T)}+C\left(\|v\|_{X(I)}^3+\|F\|_{Y(I)}^3\right) \\
& \leq 2 \delta
\end{aligned}
$$
and
$$
\begin{aligned}
\left\|\Phi\left(v_1\right)-\Phi\left(v_2\right)\right\|_{X(I)} & \leq C\left\|\left|F+v_1\right|^2\left(F+v_1\right)-\left|F+v_2\right|^2\left(F+v_2\right)\right\|_{G(I)} \\
& \leq C\left\|v_1-v_2\right\|_{X(I)}\left(\left\|v_1\right\|_{X(I)}^2+\left\|v_2\right\|_{X(I)}^2+\|F\|_{Y(I)}^2\right) \\
& \leq \frac{1}{2}\left\|v_1-v_2\right\|_{X(I)}.
\end{aligned}
$$
It follows that the map $\Phi: \mathcal{B} \rightarrow \mathcal{B}$ is a contraction with respect to the $X(I)$ norm and we infer the existence of a unique solution $v \in C(I ; H_x^1(\mathbb{R}^4)) \cap X(I)$ to \eqref{eq6.1}. By iterating this local well-posedness argument we conclude that the solution extends to a unique solution $v: I_* \times \mathbb{R}^4 \rightarrow \mathbb{C}$ to \eqref{eq6.1} with maximal time interval of existence $I_*$.

\textbf{Finite time blowup.} Assume $T_{+}:=\sup I_*<\infty$ and $\|v\|_{X\left(\left[0, T_{+}\right)\right)}<\infty$. We want to find a time $0<t_1<T_{+}$ such that
\begin{equation}\label{eq6.3}
  \left\|e^{i\left(t-t_1\right) H} v\left(t_1\right)\right\|_{X\left(\left[t_1, T_{+}\right)\right)}+\|F\|_{Y\left(\left[t_1, T_{+}\right)\right)} \leq \frac{\delta}{2}.
\end{equation}
Using $\|F\|_{Y(\mathbb{R})}<\infty$ and $\left\|e^{i\left(t-t_1\right) H} v\left(t_1\right)\right\|_{X\left(\left[t_1, \infty\right)\right)}<\infty$ and Proposition \ref{P3.2}, there exists $\eta>0$ such that
$$
\left\|e^{i\left(t-t_1\right) H} v\left(t_1\right)\right\|_{X\left(\left[t_1, T_{+}+\eta\right)\right)}+\|F\|_{Y\left(\left[t_1, T_{+}+\eta\right)\right)} \leq \delta.
$$
However, the solution $v(t)$ extends beyond time $T_{+}=\sup I_*$ by the local well-posedness result, which is a contradiction. Now, we prove \eqref{eq6.3}. Note that
$$
e^{i\left(t-t_1\right) H} v\left(t_1\right)=v(t) \pm i \int_{t_1}^t e^{i(t-s) H}|F+v|^2(F+v)(s) d s.
$$
Then it follows from \eqref{eq3.16} and \eqref{eq4.13} that
$$
\begin{aligned}
\left\|e^{i\left(t-t_1\right) \Delta} v\left(t_1\right)\right\|_{X\left(\left[t_1, T_{+}\right)\right.} & \leq\|v\|_{X\left(\left[t_1, T_{+}\right)\right)}+\left\|\int_{t_1}^t e^{i(t-s) \Delta}|F+v|^2(F+v)(s) d s\right\|_{X\left(\left[t_1, T_{+}\right)\right)} \\
& \leq\|v\|_{X\left(\left[t_1, T_{+}\right)\right)}+C\left(\|F\|_{Y\left(\left[t_1, T_{+}\right)\right)}^3+\|v\|_{X\left(\left[t_1, T_{+}\right)\right)}^3\right).
\end{aligned}
$$
According to the Proposition \ref{P3.2}(i) and the assumptions $\|v\|_{X\left(\left[0, T_{+}\right)\right)}<\infty$, $\|F\|_{Y(\mathbb{R})}<\infty$, we can obtain that $\|v\|_{X\left(\left[t_1, T_{+}\right)\right)} \rightarrow 0$ and $\|F\|_{Y\left(\left[t_1, T_{+}\right)\right)} \rightarrow 0$ as $t_1 \nearrow T_{+}$, which implies \eqref{eq6.3}.

\textbf{Scattering.} We prove the scattering statement for a global solution $v(t)$ to \eqref{eq6.1} satisfying $\|v\|_{X(\mathbb{R})}<\infty$. By similar arguments as above we infer that the scattering state in the positive time direction
$$
v^{+}:=v_0+i \int_0^{\infty} e^{-i s H}|F+v|^2(F+v)(s) d s
$$
belongs to $H_x^1\left(\mathbb{R}^4\right)$ and satisfies $\left\|v(t)-e^{i t H} v^{+}\right\|_{H_x^1\left(\mathbb{R}^4\right)} \rightarrow 0$ as $t \rightarrow \infty$. An analogous argument holds for the negative time direction.
\end{proof}

\begin{proof}[\bf Proof of Theorem \ref{t1.1}]
We seek a solution to \eqref{eq1.3} of the form
$$
u(t)=e^{i t H} f^\omega+v(t).
$$
To this end the nonlinear component $v(t)$ must be a solution to the following equation
\begin{equation}\label{eq6.4}
  \left(i \partial_t-H\right) v+ \left|e^{i t H} f^\omega+v\right|^2\left(e^{i t H} f^\omega+v\right)=0
\end{equation}
with zero initial data $v(0)=0$. By Lemma \ref{L5.4}, we have $\left\|e^{i t \Delta} f^\omega\right\|_{Y(\mathbb{R})}<\infty$ for almost every $\omega \in \Omega$. Hence, there exists an interval $I^\omega$ such that $\left\|e^{i t \Delta} f^\omega\right\|_{Y\left(I^\omega\right)} \leq \delta$ for almost every $\omega \in \Omega$ by Proposition \ref{P3.2}(i), where $0<\delta \ll 1$ is the small absolute constant from Proposition \ref{P6.1}. Therefore, there exists a unique solution $v \in C(I^\omega ; H_x^1(\mathbb{R}^4)) \cap X(I^\omega)$ to \eqref{eq6.4} for almost every $\omega \in \Omega$ via Proposition \ref{P6.1}.
\end{proof}

\section{Perturbation theory}
In this subsection, we establish perturbation theory to compare solutions to the forced nonlinear Schr\"odinger equation
\begin{equation}\label{eq7.1}
  \left\{\begin{aligned}
&(i \partial_t-H) v +|F+v|^2(F+v)=0, \\
&v(t_0)  =v_0 \in H_x^1(\mathbb{R}^4)
\end{aligned}\right.
\end{equation}
with solutions to the ``usual'' nonlinear Schr\"odinger equation
\begin{equation}\label{eq7.2}
  \left\{\begin{aligned}
&(i \partial_t-H) u  +|u|^2 u=0, \\
&u(t_0)   =u_0 \in H_x^1(\mathbb{R}^4) .
\end{aligned}\right.
\end{equation}
\begin{theorem}(Short-time perturbations)\label{t7.1}
Let $I \subset \mathbb{R}$ be a time interval with $t_0 \in I$ and let $v_0, u_0 \in H_x^1\left(\mathbb{R}^4\right)$. There exist small absolute constants $0<\delta \ll 1$ and $0<\eta_0 \ll 1$ with the following properties. Let $u: I \times \mathbb{R}^4 \rightarrow \mathbb{C}$ be the solution to \eqref{eq7.2} with initial data $u\left(t_0\right)=u_0$ satisfying
$$
\|u\|_{X(I)} \leq \delta
$$
and let $F: I \times \mathbb{R}^4 \rightarrow \mathbb{C}$ be a forcing term such that
$$
\|F\|_{Y(I)} \leq \eta
$$
for some $0<\eta \leq \eta_0$. Suppose also that
$$
\left\|v_0-u_0\right\|_{H_x^1\left(\mathbb{R}^4\right)} \leq \eta_0.
$$
Then there exists a unique solution $v: I \times \mathbb{R}^4 \rightarrow \mathbb{C}$ to \eqref{eq7.1} with initial data $v\left(t_0\right)=v_0$ and we have
\begin{equation}\label{eq7.3}
  \|v-u\|_{L_t^{\infty} H_x^1\left(I \times \mathbb{R}^4\right)}+\|v-u\|_{X(I)} \leq C_0\left(\left\|v_0-u_0\right\|_{H_x^1\left(\mathbb{R}^4\right)}+\eta\right)
\end{equation}
for some absolute constant $C_0 \geq 1$.
\end{theorem}
\begin{proof}
Define $w:=v-u$ and observe that $w$ is a solution to the difference equation
$$
\left\{\begin{aligned}
&\left(i \partial_t-H\right) w +|F+u+w|^2(F+u+w)-|u|^2 u=0 \text { on } I \times \mathbb{R}^4, \\
&w\left(t_0\right)   =v_0-u_0 .
\end{aligned}\right.
$$
By the linear estimate \eqref{eq3.16} and the nonlinear estimate \eqref{eq4.15}, we find that
\begin{eqnarray*}
&& \|w\|_{L_t^{\infty} H_x^1\left(I \times \mathbb{R}^4\right)}+\|w\|_{X(I)} \\
&\lesssim&\left\|v_0-u_0\right\|_{H_x^1\left(\mathbb{R}^4\right)}+\|F\|_{Y(I)}^3+\|w\|_{X(I)}^3+\|u\|_{X(I)}^2\|F\|_{Y(I)}+\|u\|_{X(I)}^2\|w\|_{X(I)} \\
& \lesssim&\left\|v_0-u_0\right\|_{H_x^1\left(\mathbb{R}^4\right)}+\eta^3+\|w\|_{X(I)}^3+\delta^2 \eta+\delta^2\|w\|_{X(I)} .
\end{eqnarray*}
The assertion now follows from a standard continuity argument.
\end{proof}
\begin{theorem}(Long-time perturbations)\label{t7.2}
Let $I \subset \mathbb{R}$ be a time interval with $t_0 \in I$ and let $v_0 \in H_x^1(\mathbb{R}^4)$. Let $u: I \times \mathbb{R}^4 \rightarrow \mathbb{C}$ be the solution to \eqref{eq7.2} with initial data $u(t_0)=v_0$ satisfying
$$
\|u\|_{X(I)} \leq K.
$$
Then there exists $0<\eta_1(K) \ll 1$ such that for any forcing term $F: I \times \mathbb{R}^4 \rightarrow \mathbb{C}$ satisfying
$$
\|F\|_{Y(I)} \leq \eta
$$
for some $0<\eta \leq \eta_1(K)$, there exists a unique solution $v: I \times \mathbb{R}^4 \rightarrow \mathbb{C}$ to \eqref{eq7.1} with initial data $v(t_0)=v_0$ and it holds that
$$
\|v-u\|_{L_t^{\infty} H_x^1(I \times \mathbb{R}^4)}+\|v-u\|_{X(I)} \lesssim e^{C_1 K^{\frac{4}{\varepsilon}}} \eta
$$
for some absolute constant $C_1 \gg 1$. In particular, it holds that
$$
\|v-u\|_{L_t^{\infty} H_x^1\left(I \times \mathbb{R}^4\right)}+\|v-u\|_{X(I)} \lesssim 1.
$$
\end{theorem}
\begin{proof}
Without loss of generality, we assume that $t_0=\inf I$. We first use the time divisibility property of the $X(I)$ norm to partition the interval $I$ into $J \equiv J(K)$ consecutive intervals $I_j$, $j=1, \ldots, J$, with disjoint interiors such that
$$
\|u\|_{X\left(I_j\right)} \leq \delta
$$
for $j=1, \ldots, J$, where $0<\delta \ll 1$ is the constant from Theorem \ref{t7.1}. Denote $t_{j-1}:=\inf I_j$ for $j=1, \ldots, J$. We would like to apply Theorem \ref{t7.1} on each interval $I_j$ to infer bounds on the $X(I_j)$ norm of $v-u$. To this end we have to make sure that for $j=1, \ldots, J$, it holds that
\begin{equation}\label{eq7.4}
  \|F\|_{Y\left(I_j\right)} \leq \eta_0
\end{equation}
and
\begin{equation}\label{eq7.5}
  \left\|v\left(t_{j-1}\right)-u\left(t_{j-1}\right)\right\|_{H_x^1\left(\mathbb{R}^4\right)} \leq \eta_0.
\end{equation}
where $0<\eta_0 \ll 1$ is the constant from Theorem \ref{t7.1}. Below we will in particular choose $0<\eta_1(K) \leq \eta_0$ which takes care of \eqref{eq7.4}. To ensure \eqref{eq7.5}, by using induction, it follows that
\begin{equation}\label{eq7.6}
  \|v-u\|_{L_t^{\infty} H_x^1\left(I_j \times \mathbb{R}^4\right)}+\|v-u\|_{X\left(I_j\right)} \leq\left(2 C_0\right)^j \eta
\end{equation}
for $j=1, \ldots, J$, if we choose $0<\eta_1(K) \ll 1$ sufficiently small depending on the size of $K$. Obviously, \eqref{eq7.5} holds for $j=1$, we obtain \eqref{eq7.6} for the case $j=1$ from an application of \eqref{eq7.3}. Now suppose that \eqref{eq7.6} holds for all $1 \leq i \leq j-1$ and suppose that
$$
\left(2 C_0\right)^{j-1} \eta \leq \eta_0,
$$
then we can prove that \eqref{eq7.6} also holds for $j$. By the inductive hypothesis we can apply \eqref{eq7.3} on the interval $I_j$ and obtain that
\begin{eqnarray*}
\|v-u\|_{L_t^{\infty} H_x^1\left(I_j \times \mathbb{R}^4\right)}+\|v-u\|_{X\left(I_j\right)} & \leq& C_0\left(\left\|v\left(t_{j-1}\right)-u\left(t_{j-1}\right)\right\|_{H_x^1\left(\mathbb{R}^4\right)}+\eta\right) \\
& \leq& C_0\left(\left(2 C_0\right)^{j-1} \eta+\eta\right) \\
& \leq&\left(2 C_0\right)^j \eta,
\end{eqnarray*}
which yields \eqref{eq7.6} for $j$. In order to complete the induction, we need to ensure
$(2 C_0)^{j-1} \eta \leq \eta_0$ for $j=1, \ldots, J$. In fact, by Proposition \ref{P3.2}(ii), we have that
$$
J \sim \frac{\|u\|_{X(I)}^{\frac{4}{\varepsilon}}}{\delta^{\frac{4}{\varepsilon}}} \lesssim K^{\frac{4}{\varepsilon}} .
$$
Hence, it suffices to fix
$$
\eta_1(K):=e^{-C_1 K^{\frac{4}{\varepsilon}}} \eta_0
$$
for some large absolute constant $C_1 \gg 1$. Finally, we sum up the bounds \eqref{eq7.6} to obtain that
\begin{eqnarray*}
\|v-u\|_{L_t^{\infty} H_x^1\left(I \times \mathbb{R}^4\right)}+\|v-u\|_{X(I)} & \leq& \sum_{j=1}^J\|v-u\|_{L_t^{\infty} H_x^1\left(I_j \times \mathbb{R}^4\right)}+\|v-u\|_{X\left(I_j\right)} \\
& \leq& \sum_{j=1}^J\left(2 C_0\right)^j \eta \\
& \lesssim&\left(2 C_0\right)^J \eta \\
& \lesssim& e^{C_1 K^{\frac{4}{\varepsilon}}} \eta .
\end{eqnarray*}
This establishes assertions follow from the choice of $\eta_1(K)$.
\end{proof}
\section{Orbital stability}
In this section, we consider the following nonlinear Schr\"odinger equation with mixed power nonlinearities
\begin{equation}\label{eq8.1}
  i \partial_t u+\Delta u-V(x)u +|u|^{q-2}u+|u|^{2^*-2}u=0,\ (x, t) \in \mathbb{R}^d \times \mathbb{R},
\end{equation}
where  $d\geq4, u:\mathbb{R}^d\times\mathbb{R}\rightarrow\mathbb{C}, \mu>0, 2<q<2+\frac{4}{d}$ and $2^*=\frac{2d}{d-2}$. Obviously, we can obtain global well-posedness results for \eqref{eq8.1}, see \cite{JWZY2024}. Now, we are interested in almost sure orbital stability of ground state standing waves of prescribed mass for \eqref{eq8.1}.

We recall that standing waves to \eqref{eq8.1} are solutions of the form $v(x,t)=e^{-i \lambda t}u(x),\ \lambda\in\mathbb{R}$. Then the function $u(x)$ satisfies the equation
\begin{equation}\label{eq8.2}
  -\Delta u+V(x)u-\lambda u=|u|^{q-2}u+|u|^{2^*-2}u,\ x \in \mathbb{R}^d .
\end{equation}
One can search for solutions to \eqref{eq8.2} having a prescribed $L^2$-norm. Defining on $H^1(\mathbb{R}^d, \mathbb{R})$ the energy functional
\begin{equation*}
  I(u)=\frac{1}{2}\|\nabla u\|_2^2+\frac{1}{2}\int_{\mathbb{R}^d}V(x)u^2dx-\frac{1}{q}\|u\|_q^q-\frac{1}{2^*}\|u\|_{2^*}^{2^*},
\end{equation*}
it is standard to check that $I(u)$ is of class $C^1$ and that a critical point of $I_{\mu}$ restricted to the (mass) constraint
\begin{equation*}
  S(a) := \{u\in H^1(\mathbb{R}^d, \mathbb{R}) : \|u\|_2^2 = a\}
\end{equation*}
gives rise to a solution to \eqref{eq8.2}, satisfying $\|u\|_2^2 = a$.
\subsection{Ground state solution}
We shall make use of the following classical inequalities:

\textbf{Sobolev inequality}(see \cite{HB1983}) For any $d \geq 4$ there exists an optimal constant $\mathcal{S}>0$ depending only on $d$, such that
\begin{equation}\label{eq8.3}
  \mathcal{S}\|u\|_{2^*}^2 \leq\|\nabla u\|_2^2,\ \forall u \in H^1(\mathbb{R}^d).
\end{equation}

\textbf{Gagliardo-Nirenberg inequality}(see \cite{LN1959}) If $d \geq 2, p \in[2, \frac{2d}{d-2})$ and $\beta=d(\frac{1}{2}-\frac{1}{p})$, then
\begin{equation}\label{eq8.4}
  \|u\|_p \leq C_{d, p}\|\nabla u\|_2^\beta\|u\|_2^{1-\beta},
\end{equation}
for all $u \in H^1(\mathbb{R}^d)$.

Now, it follows from \eqref{eq8.4} that
\begin{eqnarray*}
I(u)&=&\frac{1}{2}\|\nabla u\|_2^2+\frac{1}{2}\int_{\mathbb{R}^d}V(x)u^2dx-\frac{1}{q}\|u\|_q^q-\frac{1}{2^*}\|u\|_{2^*}^{2^*} \\
&\geq&\frac{1}{2}\left(1-\left\|V_{-}\right\|_{\frac{d}{2}} \mathcal{S}^{-1}\right) \int_{\mathbb{R}^d}|\nabla u|^2 d x-\frac{C_{d, q}^q a^{\frac{2 q-d(q-2)}{4}}}{q}\left(\int_{\mathbb{R}^d}|\nabla u|^2 d x\right)^{\frac{d(q-2)}{4}}\nonumber\\
&&-\frac{1}{2^*\cdot \mathcal{S}^{
\frac{2^*}{2}}}\left(\int_{\mathbb{R}^d}|\nabla u|^2 d x\right)^{\frac{2^*}{2}},
\end{eqnarray*}
we consider the function $f(a,\rho)$ defined on $(0, \infty)\times(0, \infty)$ by
$$
f(a,\rho):=\frac{1}{2}\left(1-\left\|V_{-}\right\|_{\frac{d}{2}} \mathcal{S}^{-1}\right) -\frac{C_{d, q}^q a^{\frac{2 q-d(q-2)}{4}}}{q}\rho^{\frac{d(q-2)-4}{4}}-\frac{1}{2^*\cdot \mathcal{S}^{
\frac{2^*}{2}}}\rho^{\frac{2^*-2}{2}}.
$$
For each $a\in(0, \infty)$, its restriction $f_a(\rho)$ defined on $(0, \infty)$ by $f_a(\rho):=f(a, \rho)$.
\begin{lemma}\label{L8.1}
For each $a>0$, the function $f(a,\rho)$ has a unique global maximum and the maximum value satisfies
$$
\begin{cases}\max\limits_{\rho>0} f_a(\rho)>0, & \text { if } \quad a<a_0, \\ \max\limits_{\rho>0}f_a(\rho)=0 & \text { if } \quad a=a_0, \\ \max\limits_{\rho>0}f_a(\rho)<0, & \text { if } \quad a>a_0,\end{cases}
$$
where
$$
a_0:=\left[\frac{ \frac{1}{2}\left(1-\left\|V_{-}\right\|_{\frac{d}{2}} \mathcal{S}^{-1}\right)}{2 K}\right]^{\frac{d}{2}}>0
$$
with
\begin{eqnarray*}
K &:=&\frac{C_{d, q}^q }{q} \left[ \frac{d[4-d(q-2)]C_{d,q}^q\mathcal{S}^{\frac{2^*}{2}}}{4q}\right]^{\frac{d(q-2)-4}{2\cdot2^*-d(q-2)}} +\frac{1}{2^* \mathcal{S}^{\frac{2^*}{2}}}\left[ \frac{d[4-d(q-2)]C_{d,q}^q\mathcal{S}^{\frac{2^*}{2}}}{4q}\right]^{\frac{2(2^*-2)}{2\cdot2^*-d(q-2)}}\\
&>&0.
\end{eqnarray*}
\end{lemma}
\begin{proof}
By definition of $f_a(\rho)$, we have that
$$
f_a^{\prime}(\rho)=-\frac{C_{d, q}^q a^{\frac{2 q-d(q-2)}{4}}(d(q-2)-4)}{4q}\rho^{\frac{d(q-2)-8}{4}}-\frac{2^*-2}{2}\cdot\frac{1}{2^*\cdot \mathcal{S}^{
\frac{2^*}{2}}}\rho^{\frac{2^*-4}{2}}.
$$
Hence, the equation $f^{\prime}(\rho)=0$ has a unique solution given by
$$
\rho_a=\left[ \frac{d[4-d(q-2)]C_{d,q}^q\mathcal{S}^{\frac{2^*}{2}}}{4q}\right]^{\frac{4}{2\cdot2^*-d(q-2)}} a^{\frac{2q-d(q-2)}{2\cdot2^*-d(q-2)}} .
$$
Taking into account that $f_a(\rho) \rightarrow-\infty$ as $\rho \rightarrow 0$ and $f_a(\rho) \rightarrow-\infty$ as $\rho \rightarrow \infty$, we obtain that $\rho_a$ is the unique global maximum point of $f(a,\rho)$ and the maximum value is
\begin{eqnarray*}
\max_{\rho>0}f_a(\rho) &=&\frac{1}{2}\left(1-\left\|V_{-}\right\|_{\frac{d}{2}} \mathcal{S}^{-1}\right) -\frac{C_{d, q}^q a^{\frac{2 q-d(q-2)}{4}}}{q} \left[ \frac{d[4-d(q-2)]C_{d,q}^q\mathcal{S}^{\frac{2^*}{2}}}{4q}\right]^{\frac{d(q-2)-4}{2\cdot2^*-d(q-2)}}    \\
&&\cdot a^{\frac{2q-d(q-2)}{2\cdot2^*-d(q-2)}\cdot\frac{d(q-2)-4}{4}}-\frac{a^{\frac{2q-d(q-2)}{2\cdot2^*-d(q-2)}\cdot\frac{2^*-2}{2}}}{2^*\cdot \mathcal{S}^{\frac{2^*}{2}}}\left[ \frac{d[4-d(q-2)]C_{d,q}^q\mathcal{S}^{\frac{2^*}{2}}}{4q}\right]^{\frac{2(2^*-2)}{2\cdot2^*-d(q-2)}} \\
&=& \frac{1}{2}\left(1-\left\|V_{-}\right\|_{\frac{d}{2}} \mathcal{S}^{-1}\right) -\frac{C_{d, q}^q }{q} \left[ \frac{d[4-d(q-2)]C_{d,q}^q\mathcal{S}^{\frac{2^*}{2}}}{4q}\right]^{\frac{d(q-2)-4}{2\cdot2^*-d(q-2)}}a^{\frac{2}{d}}    \\
&&-\frac{1}{2^*\cdot \mathcal{S}^{\frac{2^*}{2}}}\left[ \frac{d[4-d(q-2)]C_{d,q}^q\mathcal{S}^{\frac{2^*}{2}}}{4q}\right]^{\frac{2(2^*-2)}{2\cdot2^*-d(q-2)}}a^{\frac{2}{d}} \\
&=& \frac{1}{2}\left(1-\left\|V_{-}\right\|_{\frac{d}{2}} \mathcal{S}^{-1}\right)-Ka^{\frac{2}{d}}
\end{eqnarray*}
By the definition of $a_0$, we have that $\max\limits_{\rho>0} f_{a_0}(\rho)=0$, and hence the lemma follows.
\end{proof}
\begin{lemma}\label{L8.2}
Let $(a_1, \rho_1) \in(0, \infty) \times(0, \infty)$ be such that $f(a_1, \rho_1) \geq 0$. Then for any $a_2 \in(0, a_1]$, it holds
$$
f(a_2, \rho_2) \geq 0  \text { if }  \rho_2 \in\left[\frac{a_2}{a_1} \rho_1, \rho_1\right].
$$
\end{lemma}
\begin{proof}
Note that $a \rightarrow f(a, \rho)$ is a non-increasing function, so we have
\begin{equation}\label{eq8.5}
  f(a_2, \rho_1) \geq f(a_1, \rho_1) \geq 0.
\end{equation}
Now, by direct calculations, one has
\begin{eqnarray}\label{eq8.6}
f(a_2, \frac{a_2}{a_1} \rho_1)&=&\frac{1}{2}\left(1-\left\|V_{-}\right\|_{\frac{d}{2}} \mathcal{S}^{-1}\right) -\frac{C_{d, q}^q a_2^{\frac{2 q-d(q-2)}{4}}}{q}\left(\frac{a_2}{a_1} \rho_1\right)^{\frac{d(q-2)-4}{4}}-\frac{1}{2^*\cdot \mathcal{S}^{
\frac{2^*}{2}}}\left(\frac{a_2}{a_1} \rho_1\right)^{\frac{2^*-2}{2}}\nonumber\\
&\geq&\frac{1}{2}\left(1-\left\|V_{-}\right\|_{\frac{d}{2}} \mathcal{S}^{-1}\right) -\frac{C_{d, q}^q a_1^{\frac{2 q-d(q-2)}{4}}}{q}\rho_1^{\frac{d(q-2)-4}{4}}-\frac{1}{2^*\cdot \mathcal{S}^{
\frac{2^*}{2}}}\rho_1^{\frac{2^*-2}{2}}  \nonumber\\
&=&f(a_1, \rho_1)\nonumber\\
&\geq& 0
\end{eqnarray}
since $q-2>0$. We claim that if $f_{a_2}\left(\rho^{\prime}\right) \geq 0$ and $f_{a_2}\left(\rho^{\prime \prime}\right) \geq 0$, then
\begin{equation}\label{eq8.7}
  f(a_2, \rho)=g_{a_2}(\rho) \geq 0 \text { for any } \rho \in\left[\rho^{\prime}, \rho^{\prime\prime}\right] .
\end{equation}
In fact, if $f_{a_2}(\rho)<0$ for some $\rho \in (\rho^{\prime}, \rho^{\prime \prime} )$, then there exists a local minimum point on $(\rho_1, \rho_2)$ and this contradicts the fact that the function $g_2(\rho)$ has a unique critical point which has to coincide necessarily with its unique global maximum via Lemma \ref{L8.1}. Now, we can choose $\rho^{\prime}=\frac{a_2}{a_1}\rho_1$ and $\rho^{\prime \prime}=\rho_1$, then \eqref{eq8.5}-\eqref{eq8.7} imply the lemma.
\end{proof}
\begin{lemma}\label{L8.3}
For any $u \in S(a)$, we have that
$$
I(u) \geq\|\nabla u\|_2^2 f\left(a,\|\nabla u\|_2^2\right) .
$$
\end{lemma}
\begin{proof}
According to the definition of $f(a,\rho)$, the lemma clearly holds.
\end{proof}
Note that by Lemma \ref{L8.1} and Lemma \ref{L8.2}, we have that $f(a_0, \rho_0)=0$ and $f(a, \rho_0)>0$ for all $a \in(0, a_0)$. We define
$$
\mathbf{B}_{\rho_0}:=\left\{u \in H^1(\mathbb{R}^d):\|\nabla u\|_2^2<\rho_0\right\} \text { and } \mathbf{V}(a):=S(a) \cap \mathbf{B}_{\rho_0} .
$$
Now, we consider the following local minimization problem: for any $a \in(0, a_0)$,
$$
m(a):=\inf_{u \in \mathbf{V}(a)} I(u).
$$
\begin{lemma}\label{L8.4}
For any $a \in\left(0, a_0\right)$, it holds that
$$
m(a)=\inf _{u \in \mathbf{V}(a)} I(u)<0<\inf _{u \in \partial \mathbf{V}(a)} I(u) .
$$
\end{lemma}
\begin{proof}
For any $u \in \partial \mathbf{V}(a)$, we have $\|\nabla u\|_2^2=\rho_0$. By using Lemma \ref{L8.3}, we get
$$
I(u) \geq\|\nabla u\|_2^2 f(\|u\|_2^2,\|\nabla u\|_2^2)=\rho_0 f(a, \rho_0)>0.
$$
Now let $u \in S(a)$ be arbitrary but fixed. For $s \in(0, \infty)$ we set
$$
u_s(x):=s^{\frac{d}{2}} u(s x) .
$$
Obviously, $u_s \in S(a)$ for any $s \in(0, \infty)$. We define on $(0, \infty)$ the map,
\begin{eqnarray*}
\psi_u(s):=I(u_s)&=&\frac{1}{2}\|\nabla u_s\|_2^2+\frac{1}{2}\int_{\mathbb{R}^d}V(x)u_s^2dx-\frac{1}{q}\|u_s\|_q^q-\frac{1}{2^*}\|u_s\|_{2^*}^{2^*}\\
&\leq&\frac{1}{2}\left(1+\|V\|_{\frac{d}{2}} \mathcal{S}^{-1}\right) \|\nabla u_s\|_2^2-\frac{1}{q}\|u_s\|_q^q-\frac{1}{2^*}\|u_s\|_{2^*}^{2^*}\\
&=&\frac{s^2}{2}\left(1+\|V\|_{\frac{d}{2}} \mathcal{S}^{-1}\right)\|\nabla u\|_2^2 -\frac{1}{q} s^{\frac{d(q-2)}{2}}\|u\|_q^q-\frac{s^{2^*}}{2^*}\|u\|_{2^*}^{2^*}.
\end{eqnarray*}
Note that $\frac{d(q-2)}{2}<2$ and $2^*>2$, we see that $\psi_u(s) \rightarrow 0^{-}$ as $s \rightarrow 0$. Hence, there exists $s_0>0$ small enough such that $\left\|\nabla(u_{s_0})\right\|_2^2=s_0^2\|\nabla u\|_2^2<\rho_0$ and $I(u_*)=\psi_\mu(s_0)<0$, which implies that $m(a)<0$.
\end{proof}
We now introduce the set
$$
\mathcal{M}_a:=\left\{u \in \mathbf{V}(a): I(u)=m(a)\right\}
$$
and we collect some properties of $m(a)$.
\begin{lemma}\label{L8.5}
$\mathrm{(i)}$ $a \in\left(0, a_0\right) \mapsto m(a)$ is a continuous mapping.

$\mathrm{(ii)}$ Let $a \in\left(0, a_0\right)$. We have for all $\alpha \in(0, a): m(a) \leq m(\alpha)+m(a-\alpha)$ and if $m(\alpha)$ or $m(a-\alpha)$ is reached then the inequality is strict.
\end{lemma}
\begin{proof}
(i) Let $a \in(0, a_0)$ be arbitrary and $(a_n) \subset(0, a_0)$ be such that $a_n \rightarrow a$. From the definition of $m\left(a_n\right)$ and since $m\left(a_n\right)<0$(see Lemma \ref{L8.4}), for any $\varepsilon>0$ sufficiently small, there exists $u_n \in \mathbf{V}\left(a_n\right)$ such that
\begin{equation}\label{eq8.8}
  I\left(u_n\right) \leq m\left(a_n\right)+\varepsilon \text { and } I\left(u_n\right)<0 .
\end{equation}
Set $y_n:=\sqrt{\frac{a}{a_n}} u_n$ and hence $y_n \in S(a)$. We have that $y_n \in \mathbf{V}(a)$. In fact, if $a_n \geq a$, then
$$
\left\|\nabla y_n\right\|_2^2=\frac{a}{a_n}\left\|\nabla u_n\right\|_2^2 \leq\left\|\nabla u_n\right\|_2^2<\rho_0 .
$$
If $a_n<a$, we have $f\left(a_n, \rho\right) \geq 0$ for any $\rho \in\left[\frac{a_n}{a} \rho_0, \rho_0\right]$ by Lemma \ref{L8.2}. Hence, it follows from Lemma \ref{L8.3} and \eqref{eq8.8} that $f\left(a_n,\left\|\nabla u_n\right\|_2^2\right)<0$, thus $\left\|\nabla u_n\right\|_2^2<\frac{a_n}{a} \rho_0$ and
$$
\left\|\nabla y_n\right\|_2^2=\frac{a}{a_n}\left\|\nabla u_n\right\|_2^2<\frac{a}{a_n}\cdot\frac{a_n}{a} \rho_0=\rho_0 .
$$
Since $y_n \in \mathbf{V}(a)$, we can write
$$
m(a) \leq I\left(y_n\right)=I\left(u_n\right)+\left[I\left(y_n\right)-I\left(u_n\right)\right],
$$
where
\begin{eqnarray*}
&&I\left(y_n\right)-I\left(u_n\right)\\
&=&-\frac{1}{2}\left(\frac{a}{a_n}-1\right)\left\|\nabla u_n\right\|_2^2-\frac{1}{q}\left[\left(\frac{a}{a_n}\right)^{\frac{q}{2}}-1\right]\left\|u_n\right\|_q^q-\frac{1}{2^*}\left[\left(\frac{a}{a_n}\right)^{\frac{2^*}{2}}-1\right]\left\|u_n\right\|_{2^*}^{2^*}\\
&&+\frac{1}{2}\left[\frac{a}{a_n}-1\right]\int_{\mathbb{R}^d}V(x)u_n^2dx.
\end{eqnarray*}
Since $\left\|\nabla u_n\right\|_2^2<\rho_0$, then $\left\|u_n\right\|_q^q$ and $\left\|u_n\right\|_{2^*}^{2^*}$ are uniformly bounded. Hence, we have
\begin{equation}\label{eq8.9}
  m(a) \leq I\left(y_n\right)=I\left(u_n\right)+o_n(1)
\end{equation}
as $n \rightarrow \infty$. Combining \eqref{eq8.8} and \eqref{eq8.8}, we have
$$
m(a) \leq m\left(a_n\right)+\varepsilon+o_n(1).
$$
Now, let $u \in \mathbf{V}(a)$ be such that
$$
I(u) \leq m(a)+\varepsilon  \text { and }   I(u)<0 .
$$
Set $u_n:=\sqrt{\frac{a_n}{a}} u$ and hence $u_n \in S\left(a_n\right)$. Clearly, $\|\nabla u\|_2^2<\rho_0$ and $a_n \rightarrow a$ imply $\left\|\nabla u_n\right\|_2^2<\rho_0$ for $n$ large enough, so that $u_n \in \mathbf{V}\left(a_n\right)$ and $I\left(u_n\right) \rightarrow I(u)$. So we have
$$
m\left(a_n\right) \leq I\left(u_n\right)=I(u)+\left[I\left(u_n\right)-I(u)\right] \leq m(a)+\varepsilon+o_n(1) .
$$
Therefore, since $\varepsilon>0$ is arbitrary, we deduce that $m\left(a_n\right) \rightarrow m(a)$.

(ii) Fixed $\alpha \in(0, a)$, it is sufficient to prove that the following holds
\begin{equation}\label{eq8.10}
  \forall \theta \in\left(1, \frac{a}{\alpha}\right]: m(\theta \alpha) \leq \theta m(\alpha)
\end{equation}
and if $m(\alpha)$ is reached, the inequality is strict. In fact, if \eqref{eq8.10} holds then we have
\begin{eqnarray*}
m(a) & =&\frac{a-\alpha}{a} m(a)+\frac{\alpha}{a} m(a)\\
&=&\frac{a-\alpha}{a} m\left(\frac{a}{a-\alpha}(a-\alpha)\right)+\frac{\alpha}{a} m\left(\frac{a}{\alpha} \alpha\right) \\
& \leq& m(a-\alpha)+m(\alpha)
\end{eqnarray*}
with a strict inequality if $m(\alpha)$ is reached. In order to obtain that \eqref{eq8.10}, note that in view of Lemma \ref{L8.4}, there exists a $u \in V(\alpha)$ such that
\begin{equation}\label{eq8.11}
I(u) \leq m(\alpha)+\varepsilon \text { and } I(u)<0
\end{equation}
for any $\varepsilon>0$ sufficiently small. By using Lemma \ref{L8.2}, $f(\alpha, \rho) \geq 0$ for any $\rho \in\left[\frac{\alpha}{a} \rho_0, \rho_0\right]$. Hence, it follows from Lemma \ref{L8.3} and \eqref{eq8.11} that
\begin{equation}\label{eq8.12}
\|\nabla u\|_2^2<\frac{\alpha}{a} \rho_0 .
\end{equation}
Now, consider $v=\sqrt{\theta} u$. We first note that $\|v\|_2^2=\theta\|u\|_2^2=\theta \alpha$ and also, $\|\nabla v\|_2^2=\theta\|\nabla u\|_2^2< \rho_0$ because of \eqref{eq8.12}. Thus $v \in \mathbf{V}(\theta \alpha)$ and we can write
\begin{eqnarray*}
m(\theta \alpha)  \leq I(v)&=&\frac{1}{2} \theta\|\nabla u\|_2^2+\frac{1}{2}\theta\int_{\mathbb{R}^d}V(x)u^2dx-\frac{1}{q} \theta^{\frac{q}{2}}\|u\|_q^q-\frac{1}{2^*} \theta^{\frac{2^*}{2}}\|u\|_{2^*}^{2^*} \\
& <&\frac{1}{2} \theta\|\nabla u\|_2^2+\frac{1}{2}\theta\int_{\mathbb{R}^d}V(x)u^2dx-\frac{1}{q} \theta\|u\|_q^q-\frac{1}{2^*} \theta\|u\|_{2^*}^{2^*} \\
& =&\theta\left(\frac{1}{2}\|\nabla u\|_2^2+\frac{1}{2}\int_{\mathbb{R}^d}V(x)u^2dx-\frac{1}{q}\|u\|_q^q-\frac{1}{2^*}\|u\|_{2^*}^{2^*}\right)\\
&=&\theta I(u) \leq \theta(m(\alpha)+\varepsilon).
\end{eqnarray*}
Since $\varepsilon>0$ is arbitrary, we have that $m(\theta \alpha) \leq \theta m(\alpha)$. If $m(\alpha)$ is reached then we can let $\varepsilon=0$ in \eqref{eq8.11} and thus the strict inequality follows.
\end{proof}
\begin{lemma}\label{L8.6}
Let $\left(v_n\right) \subset \mathbf{B}_{\rho_0}$ be such that $\left\|v_n\right\|_q \rightarrow 0$. Then there exists a $\beta_0>0$ such that
$$
I\left(v_n\right) \geq \beta_0\left\|\nabla v_n\right\|_2^2+o_n(1).
$$
\end{lemma}
\begin{proof}
Indeed, using the Sobolev inequality, we have
\begin{eqnarray*}
I\left(v_n\right) & =&\frac{1}{2}\|\nabla v_n\|_2^2+\frac{1}{2}\int_{\mathbb{R}^d}V(x)v_n^2dx-\frac{1}{2^*}\|v_n\|_{2^*}^{2^*}+o_n(1) \\
& \geq& \frac{1}{2}\left(1-\left\|V_{-}\right\|_{\frac{d}{2}} \mathcal{S}^{-1}\right)\left\|\nabla v_n\right\|_2^2-\frac{1}{2^*} \frac{1}{\mathcal{S}^{\frac{2^*}{2}}}\left\|\nabla v_n\right\|_2^{2^*}+o_n(1) \\
& =&\left\|\nabla v_n\right\|_2^2\left[\frac{1}{2}\left(1-\left\|V_{-}\right\|_{\frac{d}{2}} \mathcal{S}^{-1}\right)-\frac{1}{2^*} \frac{1}{\mathcal{S}^{\frac{2^*}{2}}}\left\|\nabla v_n\right\|_2^{2^*-2}\right]+o_n(1) \\
& \geq&\left\|\nabla v_n\right\|_2^2\left[\frac{1}{2}\left(1-\left\|V_{-}\right\|_{\frac{d}{2}} \mathcal{S}^{-1}\right)-\frac{1}{2^*} \frac{1}{\mathcal{S}^{\frac{2^*-2}{2}}} \rho_0^{\frac{2^*-2}{2}}\right]+o_n(1).
\end{eqnarray*}
Now, since $f\left(a_0, \rho_0\right)=0$, we have that
\begin{equation*}
  \beta_0:=\frac{1}{2}\left(1-\left\|V_{-}\right\|_{\frac{d}{2}} \mathcal{S}^{-1}\right)-\frac{1}{2^*} \frac{1}{\mathcal{S}^{\frac{2^*-2}{2}}} \rho_0^{\frac{2^*-2}{2}}=\frac{C_{d, q}^q a_0^{\frac{2 q-d(q-2)}{4}}}{q}\rho_0^{\frac{d(q-2)-4}{4}}>0,
\end{equation*}
which completes the proof.
\end{proof}
\begin{lemma}\label{L8.7}
For any $a \in\left(0, \varepsilon_0\right)$, let $\left(u_n\right) \subset \mathbf{B}_{\rho_0}$ be such that $\left\|u_n\right\|_2^2 \rightarrow a$ and $F_\mu\left(u_n\right) \rightarrow m(a)$. Then, there exist a $\beta_1>0$ and a sequence $\left(y_n\right) \subset \mathbb{R}^d$ such that
\begin{equation}\label{eq8.13}
  \int_{B\left(y_n, R\right)}\left|u_n\right|^2 d x \geq \beta_1>0  \text { for some } R>0.
\end{equation}
\end{lemma}
\begin{proof}
By contradiction, assume that \eqref{eq8.13} does not bold. Since ( $u_n$ ) $\subset B_{p_0}$ and $\left\|u_n\right\|_2^2 \rightarrow a$, the sequence $\left(u_n\right)$ is bounded in $H^1(\mathbb{R}^d)$. From Lemma I.1 in \cite{PLL1984} and $2<q<2^*$, we deduce that $\left\|u_n\right\|_q \rightarrow 0$ as $n \rightarrow \infty$. Hence, Lemma \ref{L8.6} implies that $I\left(u_n\right) \geq o_n(1)$. This contradicts the fact that $m(a)<0$ and the lemma follows.
\end{proof}
\begin{lemma}\label{L8.8}
For any $a \in\left(0, a_0\right)$, if $\left(u_n\right) \subset \mathbf{B}_{\rho_0}$ is such that $\left\|u_n\right\|_2^2 \rightarrow a$ and $I\left(u_n\right) \rightarrow m(a)$ then, up to translation, $u_{n} \xrightarrow{H^1(\mathbb{R}^d)} u \in \mathcal{M}_a$. In particular the set $\mathcal{M}_c$ is compact in $H^1(\mathbb{R}^d), u_\rho$ to translation.
\end{lemma}
\begin{proof}
It follows from Lemma \ref{L8.7} and Rellich compactness theorem that there exists a sequence $\left(y_n\right) \subset \mathbb{R}^d$ such that
$$
u_n\left(x-y_n\right) \rightharpoonup u_a \neq 0 \text { in } H^1(\mathbb{R}^d).
$$
Our aim is to prove that $w_n(x):=u_n\left(x-y_n\right)-u_a(x) \rightarrow 0$ in $H^1(\mathbb{R}^d)$. Clearly,
\begin{eqnarray*}
\left\|u_n\right\|_2^2 & =&\left\|u_n\left(x-y_n\right)\right\|_2^2\\
&=&\left\|u_n\left(x-y_n\right)-u_a(x)\right\|_2^2+\left\|u_a\right\|_2^2+o_n(1) \\
& =&\left\|w_n\right\|_2^2+\left\|u_a\right\|_2^2+o_n(1).
\end{eqnarray*}
Hence,
\begin{equation}\label{eq8.14}
  \left\|w_n\right\|_2^2=\left\|u_n\right\|_2^2-\left\|u_a\right\|_2^2+o_n(1)=a-\left\|u_a\right\|_2^2+o_n(1).
\end{equation}
By a similar argument,
\begin{equation}\label{eq8.15}
\left\|\nabla w_n\right\|_2^2=\left\|\nabla u_n\right\|_2^2-\left\|\nabla u_a\right\|_2^2+o_n(1).
\end{equation}
By the translational invariance, it holds
\begin{equation}\label{eq8.16}
I\left(u_n\right)=I\left(u_n\left(x-y_n\right)\right)=I\left(w_n\right)+I\left(u_a\right)+o_n(1) .
\end{equation}
Next, we claim that
$$
\left\|w_n\right\|_2^2 \rightarrow 0.
$$
In fact, denote $a_1:=\left\|u_a\right\|_2^2>0$. By \eqref{eq8.14}, if we show that $a_1=a$ then the claim follows. By contradiction, we assume that $a_1<a$. Using \eqref{eq8.14} and \eqref{eq8.15}, we have $\left\|w_n\right\|_2^2 \leq a$ and $\left\|\nabla w_n\right\|_2^2 \leq\left\|\nabla u_n\right\|_2^2<\rho_0$ for $n$ large enough. Thus, we obtain that $w_n \in \mathbf{V}(\left\|w_n\right\|_2^2)$ and $I\left(w_n\right) \geq m\left(\left\|w_n\right\|_2^2\right)$. Recording that $I\left(u_n\right) \rightarrow m(a)$, in view of \eqref{eq8.16}, we have
$$
m(a)=I\left(w_n\right)+I\left(u_a\right)+o_n(1) \geq m\left(\left\|w_n\right\|_2^2\right)+I\left(u_a\right)+o_n(1).
$$
Since the map $a \mapsto m(a)$ is continuous, we deduce that
\begin{equation}\label{eq8.17}
m(a) \geq m\left(a-a_1\right)+I\left(u_a\right).
\end{equation}
We also have that $u_a \in \mathbf{V}\left(a_1\right)$ by the weak limit. This implies that $I\left(u_a\right) \geq m\left(a_1\right)$. If $I\left(u_a\right)>m\left(a_1\right)$, then it follows from \eqref{eq8.17} and Lemma \ref{L8.5}(ii) that
$$
m(a)>m\left(a-a_1\right)+m\left(a_1\right) \geq m\left(a-a_1+a_1\right)=m(a),
$$
which is impossible. Hence, we have $I\left(u_a\right)=m\left(a_1\right)$, namely $u_a$ is a local minimizer on $\mathbf{V}\left(a_1\right)$. So, using Lemma \ref{L8.5}(ii) with the strict inequality, we deduce from \eqref{eq8.17} that
$$
m(a) \geq m\left(a-a_1\right)+I\left(u_a\right)=m\left(a-a_1\right)+m\left(a_1\right)>m\left(a-a_1+a_1\right)=m(a),
$$
which is impossible. Thus, the claim follows and from \eqref{eq8.14} we obtain that $\left\|u_a\right\|_2^2=a$.

Now, in order to complete the proof, we need to show that $\left\|\nabla w_n\right\|_2^2 \rightarrow 0$. Note that, by using \eqref{eq8.15} and $u_a \neq 0$, we have $\left\|\nabla w_n\right\|_2^2 \leq\left\|\nabla u_n\right\|_2^2<\rho_0$ for $n$ large enough, so $\left(w_n\right) \subset B_{\rho_0}$ and it is bounded in $H^1(\mathbb{R}^d)$. Then by using the Gagliardo-Nirenberg inequality and recalling $\left\|w_n\right\|_2^2 \rightarrow 0$, we also have $\left\|w_n\right\|_q^q \rightarrow 0$. Thus, it follows from Lemma \ref{L8.6} that
\begin{equation}\label{eq8.18}
  I\left(w_n\right) \geq \beta_0\left\|\nabla w_n\right\|_2^2+o_n(1) \text { where } \beta_0>0 .
\end{equation}
Now we remember that
$$
I\left(u_n\right)=I\left(u_a\right)+I\left(w_n\right)+o_n(1) \rightarrow m(a) .
$$
Since $u_a \in \mathbf{V}(a)$ by weak limit, we have that $I\left(u_a\right) \geq m(a)$ and hence $I\left(w_n\right) \leq o_n(1)$. Then we get that $\left\|\nabla w_n\right\|_2^2 \rightarrow 0$ by using \eqref{eq8.18}.
\end{proof}
\begin{proof}[\bf Proof of Theorem \ref{t1.3}]
The fact that if $\left(u_n\right) \subset V(a)$ is such that $I\left(u_n\right) \rightarrow m(a)$ then, up to translation, $u_n \rightarrow u \in M_a$ in $H^1(\mathbb{R}^d)$ follows from Lemma \ref{L8.8}. In particular, it insures the existence of a minimizer for $I(u)$ on $\mathbf{V}(a)$ and this minimizer is a ground state.
\end{proof}
\subsection{Orbital stability}
In this subsection, we prove that the set $\mathcal{M}_a$ is orbitally stable.
\begin{lemma}\label{L8.9}
Let $v \in \mathcal{M}_a$. Then, for every $\varepsilon>0$ there exists $\delta>0$ such that:
$$
\forall \varphi \in H^1(\mathbb{R}^d) \text { s.t. }\|\varphi-v\|_{H^1(\mathbb{R}^d)}<\delta \Longrightarrow \sup _{t \in\left[0, +\infty\right)} \mathop{\operatorname{dist}}\limits_{H^1(\mathbb{R}^d)}\left(u_{\varphi}(t), \mathcal{M}_a\right)<\varepsilon,
$$
where $u_{\varphi}(t)$ is the global solution of \eqref{eq8.1} corresponding to the  initial data in $H^1(\mathbb{R}^d)$.
\end{lemma}
\begin{proof}
Suppose the theorem is false, there exists a decreasing sequence $\left(\delta_n\right) \subset \mathbb{R}^{+}$ converging to 0 and $\left(\varphi_n\right) \subset H^1(\mathbb{R}^d)$ satisfying
$$
\left\|\varphi_n-v\right\|_{H^1(\mathbb{R}^d)}<\delta_n,
$$
but
$$
\sup _{t \in\left[0, +\infty\right)} \mathop{\operatorname{dist}}_{H^1(\mathbb{R}^d)}\left(u_{\varphi_n}(t), \mathcal{M}_a\right)>\varepsilon_0
$$
for some $\varepsilon_0>0$. We observe that $\left\|\varphi_n\right\|_2^2 \rightarrow a$ and $I\left(\varphi_n\right) \rightarrow m(a)$ by continuity of $I(u)$. By conservation laws, for $n \in \mathbb{N}$ large enough, $u_{\varphi_n}$ will remains inside of $\mathbf{B}_{\rho_0}$ for all $t \in\left[0, +\infty\right)$. Indeed, if for some time $\bar{t}>0,\ \left\|\nabla u_{\varphi_n}(\bar{t})\right\|_2^2=\rho_0$, then in view of Lemma \ref{L8.4} we have that $I\left(u_{\varphi_n}(\bar{t})\right) \geq 0$ in contradiction with $m(a)<0$. Now let $t_n>0$ be the first time such that $\mathop{\operatorname{dist}}\limits_{H^1(\mathbb{R}^d)}\left(u_{\varphi_n}\left(t_n\right), \mathcal{M}_a\right)=\varepsilon_0$ and set $u_n:=u_{\varphi_n}\left(t_n\right)$. By conservation laws, $\left(u_n\right) \subset \mathbf{B}_{\rho_0}$ satisfies $\left\|u_n\right\|_2^2 \rightarrow a$ and $I\left(u_n\right) \rightarrow m(a)$ and thus, in view of Lemma \ref{L8.8}, it converges, up to translation, to an element of $\mathcal{M}_a$. Since $\mathcal{M}_a$ is invariant under translation this contradicts the equality $\mathop{\operatorname{dist}}\limits_{H^1(\mathbb{R}^d)}\left(u_n, \mathcal{M}_a\right)=\varepsilon_0>0$.
\end{proof}


\begin{thebibliography}{99}

\bibitem{TBAQ2023} T. Bartsch, S. Qi, W. Zou, Normalized solutions to Sch\"odinger equations with potential and inhomogeneous nonlinearities on large convex domains, Math. Ann., 2024, 1--47.



\bibitem{JB1994} J. Bourgain, Periodic nonlinear Schr\"odinger equation and invariant measures, Comm. Math. Phys., 166(1)(1994), 1--26.

\bibitem{JB1996} J. Bourgain, Invariant measures for the 2D-defocusing nonlinear Schr\"odinger equation, Comm. Math. Phys., 176(2)(1996), 421--445.


\bibitem{JBAB2014}  J. Bourgain, A. Bulut, Almost sure global well-posedness for the radial nonlinear Schr\"odinger equation on the unit ball II: the 3d case, J. Eur. Math. Soc., 16(6)(2014), 1289--1325.

\bibitem{HB1983}  H. Brezis, Analyse fonctionnelle: Th$\mathrm{\acute{e}}$orie et applications (Theory and Applications), Collection Math$\mathrm{\acute{e}}$matiques Appliqu$\mathrm{\acute{e}}$es pour la Ma$\mathrm{\hat{l}}$trise (Collection of Applied Mathematics for the Master's Degree), Masson, Paris, 1983.

\bibitem{NBII2008}  N. Burq, N. Tzvetkov, Random data Cauchy theory for supercritical wave equations. II. A global existence result, Invent. Math., 173(3)(2008), 477--496.



\bibitem{NBNT2014}  N. Burq, N. Tzvetkov, Probabilistic well-posedness for the cubic wave equation, J. Eur. Math. Soc., 16(1)(2014), 1--30.


\bibitem{NBI2008} N. Burq, N. Tzvetkov, Random data Cauchy theory for supercritical wave equations. I. Local theory, Invent. Math., 173(3)(2008), 449--475.

\bibitem{JCTO2012} J. Colliander, T. Oh, Almost sure well-posedness of the cubic nonlinear Schr\"odinger equation below $L^2(\mathbb{T})$, Duke Math. J., 161(3)(2012), 367--414.

\bibitem{YD2012} Y. Deng, Two-dimensional nonlinear Schr\"odinger equation with random radial data, Anal. PDE, 5(5)(2012), 913--960.

\bibitem{VDD2018}  V. D. Dinh, Well-posedness of nonlinear fractional Schr\"odinger and wave equations in Sobolev spaces, Int. J. Appl. Math., 31 (2018), 483--525.

\bibitem{BDJL2019} B. Dodson, J. L$\ddot{u}$hrmann, D. Mendelson, Almost sure local well-posedness and scattering for the $4D$ cubic nonlinear Schr\"odinger equation, Adv. Math., 347(2019), 619--676.

\bibitem{YH2016} Y. Hong, Scattering for a nonlinear Schr\"odinger equation with a potential, Commun. Pure Appl. Anal., 15(5)(2016), 1571--01601.

\bibitem{AICK2007} A. Ionescu, C. Kenig, Low-regularity Schr\"odinger maps, II, Global well-posedness in dimensions $d\geq3$, Comm. Math. Phys., 271(2)(2007), 523--559.

\bibitem{LJJJ2022} L. Jeanjean, J. Jendrej, T. T. Le, N. Visciglia, Orbital stability of ground states for a Sobolev critical Schr\"odinger equation, J. Math. Pures Appl., 164(2022), 158--179.

\bibitem{MKTT1998} M. Keel, T. Tao, Endpoint Strichartz estimates, Am. J. Math., 120(1998), 955--980.


\bibitem{HKDT2005} H. Koch, D. Tataru, Dispersive estimates for principally normal pseudodifferential operators, Commun. Pure Appl. Math., 58(2)(2005), 217--284.

\bibitem{PLL1984}  P. L. Lions, The concentration-compactness principle in the calculus of variations, The locally compact case, II, Ann. Inst. Henri Poincar$\mathrm{\acute{e}}$, Anal. Non Lin$\mathrm{\acute{e}}$aire, 1(4)(1984), 223--283.

\bibitem{LN1959}  L. Nirenberg, On elliptic partial differential equations, Ann. Sc. Norm. Super. Pisa, Cl. Sci., 13(3)(1959), 115--162.

\bibitem{REAC1930}  R. E. A. C. Paley, A. Zygmund, On some series of functions (1), Proc. Camb. Philos. Soc., 26(1930), 337--357.

\bibitem{OPYW2018} O. Pocovnicu, Y. Wang,  An $L^p$-theory for almost sure local well-posedness of the nonlinear Schr\"odinger equations, C. R. Math. Acad. Sci. Paris, 356(6)(2018), 637--643.

\bibitem{JSAS2023} J. Shen, A. Soffer, Y. Wu, Almost sure well-posedness and scattering of the 3D cubic nonlinear Schr\"odinger equation, Comm. Math. Phys., 397(2)(2023), 547--605.

\bibitem{NT2010}  N. Tzvetkov, Construction of a Gibbs measure associated to the periodic Benjamin-Ono equation, Probab. Theory Related Fields, 146(3-4)(2010), 481--514.


\bibitem{JWZY2024}  J. Wang, Z. Yin, Global well-posedness, scattering and blow-up for the energy-critical, Schr\" odinger equation with indefinite potential in the radial case, 2024, arXiv preprint, arXiv:2407.01588.


















\end{thebibliography}
\end{document}